\documentclass[10pt,oneside,a4paper,reqno]{amsart}  %%(fold)
\usepackage{mathrsfs}
\usepackage{amssymb, bbm, enumerate}
\usepackage{geometry}                % See geometry.pdf to learn the layout options. There are lots.
\usepackage[T1]{fontenc}

\usepackage{mathrsfs} 
\usepackage[shortlabels]{enumitem}

\usepackage[english]{babel}
\usepackage{graphicx}
\usepackage{color}
\usepackage[colorlinks=true, linkcolor=magenta, citecolor=cyan, urlcolor=blue]{hyperref}%link alle eq e ref
\usepackage{scalerel,stackengine}
\stackMath
\newcommand\reallywidehat[1]{%
\savestack{\tmpbox}{\stretchto{%
  \scaleto{%
    \scalerel*[\widthof{\ensuremath{#1}}]{\kern-.6pt\bigwedge\kern-.6pt}%
    {\rule[-\textheight/2]{1ex}{\textheight}}%WIDTH-LIMITED BIG WEDGE
  }{\textheight}% 
}{0.5ex}}%
\stackon[1pt]{#1}{\tmpbox}%
}
\renewcommand{\phi}{\varphi}

\renewcommand{\Im}{\textup{Im }}

\newcommand{\mc}[1]{\mathcal{#1}}

\def\be{\begin{equation}}
\def\ee{\end{equation}}
\def\bea{\begin{eqnarray}}
\def\eea{\end{eqnarray}}

\def\nn{\nonumber}
\def\T{\mathbb{T}}

\def\Z{\mathbb{Z}}
\def\N{\mathbb{N}}

\DeclareMathSymbol{\leqslant}{\mathalpha}{AMSa}{"36} % nicer `smaller or equal'
\DeclareMathSymbol{\geqslant}{\mathalpha}{AMSa}{"3E} % nicer `larger or equal'
\DeclareMathSymbol{\eset}{\mathalpha}{AMSb}{"3F}     % nicer `emptyset'
\renewcommand{\leq}{\;\leqslant\;}                   % redef. of < or =
\renewcommand{\geq}{\;\geqslant\;}                   % redef. of > or =

\DeclareMathOperator{\supp}{supp}

\DeclareMathOperator{\Span}{span}

\def\ie{\textit{i.e. }}

\def\a{\alpha}
\def\e{\varepsilon}
\def\d{\delta}
\def\g{\gamma}

\def\b{\beta}
\def\D{\Delta}
\def\l{\lambda}
\def\r{\rho}
\def\s{\sigma}
\def\t{\tau}
\def\k{\kappa}
\def\vs{\varsigma}

\def\R{\mathbb{R}}

\theoremstyle{plain}
\newtheorem{theorem}{Theorem}[section]
\newtheorem{lemma}[theorem]{Lemma}
\newtheorem{proposition}[theorem]{Proposition}
\newtheorem{corollary}[theorem]{Corollary}

\theoremstyle{definition}

\theoremstyle{remark}
\newtheorem{remark}[theorem]{Remark}

\numberwithin{equation}{section}

\definecolor{light}{gray}{.9}

\author{Giuseppe Genovese}
\address{Institut of Mathematics, University of Z\"urich,
Winterthurerstrasse 190
CH-8057 Z\"urich,  Switzerland}
\email{giuseppe.genovese@math.uzh.ch}

\author{Renato Luc\`a}
\address{BCAM - Basque Center for Applied Mathematics, 48009 Bilbao, Spain and Ikerbasque, Basque Foundation
for Science, 48011 Bilbao, Spain.}
\email{rluca@bcamath.org} 

\author{Nikolay Tzvetkov}
\address{Department of Mathematics (AGM), University of Cergy-Pontoise, 2, av. Adolphe Chauvin, 95302 Cergy-Pontoise Cedex, FRANCE}
\email{nikolay.tzvetkov@u-cergy.fr}

\title[Transport of Gaussian measures for Hamiltonian PDE\lowercase{s}]
{Transport of Gaussian measures with exponential cut-off for Hamiltonian PDE\lowercase{s}}

\date{\today}

\makeindex
\begin{document}
\begin{abstract}
We show that introducing an exponential cut-off on a suitable Sobolev norm facilitates the proof of quasi-invariance of Gaussian measures with respect to Hamiltonian PDE flows and allows us to establish the exact Jacobi formula for the density.
We exploit this idea in two different contexts, namely the periodic fractional Benjamin-Bona-Mahony (BBM) equation with dispersion~$\beta >1$ and the periodic one dimensional quintic defocussing nonlinear 
Schr\"odinger equation (NLS). For the BBM equation we study the transport of the cut-off Gaussian measures on 
fractional Sobolev spaces, while for the NLS equation we 
study the measures based on the modified energies introduced by Planchon-Visciglia and the third author. 
Moreover for the BBM equation we also show almost sure global well-posedness for data in~$C^\a(\T)$ for arbitrarily 
small~$\a>0$ and invariance of the Gaussian measure associated with the $H^{\beta/2}(\T)$ norm. 
 \end{abstract}

\maketitle

\section{Introduction}
\subsection{The setting}
This work fits in the line of research initiated in \cite{sigma} (inspired by the previous work \cite{DTV,TV13a,TV13b,TV14} on an integrable equation) aiming to study the transport of Gaussian measures under the flow of non integrable partial differential equations, in particular their invariance and quasi-invariance. We say that a measure $\mu$ is invariant under a (reversible) flow map $\{\Phi_t\}_{t\in \R}$ if $\mu\circ \Phi_t=\mu$ for any $t\in \R$ and it is quasi-invariant if $\mu\circ \Phi_t$ is absolutely continuous w.r.t. $\mu$. 

Our main aim here is to show that the introduction in the measure of an exponential weight suppressing large values of a suitable Sobolev norm (not necessarily conserved by the flow) is an efficient tool in the study of absolute continuity of the transported measure.
 Indeed it makes the proofs easier and allows us to give more information on the resulting Radon-Nikodim derivatives, establishing the so-called Jacobi formula for the density as a limit of finite-dimensional functions. We present the theory via two notable examples, namely the Benjamin-Bona-Mahony (BBM) 
equation with dispersion $\beta >1$ and the periodic quintic defocussing nonlinear Schr\"odinger equation (NLS), demonstrating that in this way we can improve on the previous analyses of \cite{sigma} and \cite{PTV} respectively.

The result of \cite{sigma} was extended to more involved models in \cite{deb, gauge, GLV0, GLV, FS,forl, GOTW,OS,OT1,OT2,OT3,OTT,PTV,phil, BBM2}. We 
believe that, beyond the BBM and NLS equation, the idea of an exponential cut-off introduced in the present paper may be relevant in the context of some of these works and, more generally, in the study of quasi-invariant measures for Hamiltonian PDEs.
We also refer to the recent paper \cite{BT} where quantitative quasi-invariance of certain Gaussian measures are exploited in questions of the long time behaviour of solutions for dispersive PDE's.

First of all we introduce the main objects we shall deal with, namely Gaussian measures on Sobolev spaces. We warn the reader that will use two slightly different definitions of Gaussian measures for the BBM and the NLS equations, with however similar or equal notations. 

We denote by $H^\sigma(\T)=H^\sigma$ the 
Sobolev space of 
real or complex valued functions (used respectively for the BBM and the NLS equation) equipped with the norm 
\begin{equation}\label{Def:SobNorm}
\|u\|_{H^\sigma}=\Big(\sum_{n \in \Z}(1+|n|^{2\sigma})|\hat{u}(n)|^2\Big)^{\frac{1}{2}}\,\, 
\end{equation}
(here and further $\widehat u(n)$ denotes the $n$-th Fourier coefficient of the function $u$). 
Let $\{h_n\}_{n> 0}$, $\{l_n\}_{n> 0}$ be two independent sequences of independent standard Gaussian random variables. Let $g_0$ be a standard Gaussian random variable independent on anything else and set
$$
g_n:=
\begin{cases}
\frac{1}{\sqrt{2}}(h_n+il_n) & n\in\N\\
\frac{1}{\sqrt{2}}(h_n-il_n) & -n\in\N\,.
\end{cases}
$$
Let $\beta>1$, $s\geq0$ and denote by $\g_{s}$ the Gaussian measure on $H^s$ induced by the random Fourier series 
\begin{equation}\label{Def:GaussmEasure}
\varphi_s(x)=\sum_{n\in\Z}\frac{g_n}{(1+|n|^{2s+\beta})^{\frac{1}{2}}}\,e^{inx}\,.
\end{equation}
This measure will be central in the analysis of the BBM equation. Observe that throughout the paper we will systematically omit the dependence of $\g_s$ on the parameter $\beta$ in the notations. 

For NLS we have complex solutions, so we consider a sequence of complex standard Gaussian random variables $\{g_n\}_{n\in\Z}$ and for integers $k \geq 2$ the Fourier series
\begin{equation}\label{Def:GaussmEasureNLS}
\varphi_{2k}(x)=\sum_{n\in\Z}\frac{g_n}{(1+|n|^{4k})^{\frac{1}{2}}}\,e^{inx}\,.
\end{equation} 
We indicate by $\g_{2k}$ the induced measure on $H^{2k-\frac12-}:=\bigcap_{\e>0} H^{2k-\frac12-\e}$.

\subsection{The BBM equation} 
For $\beta>1$, we consider the fractional BBM equation, posed on the one dimensional flat torus $\T := \R/2\pi \Z$:
\begin{equation}\label{BBM-gamma}
\partial_t u+\partial_t|D_x|^\beta u+\partial_x u+\partial_x (u^2)=0, \quad u(0,x)=u_0(x)\,,
\end{equation}
where $u$ is real valued and
$$
|D_x|^\beta(u)(x):=\sum_{n \neq 0}|n|^\beta \hat{u}(n)e^{inx}  \,.
$$

It turns out that the Sobolev spaces introduced above are natural for the study of the global Cauchy problem for \eqref{BBM-gamma}. This is because the $H^{\beta/2}$ norm is formally preserved by the BBM equation. 

We refer to \cite{sigma} for the modelling arguments leading to the derivation of \eqref{BBM-gamma}.
Using the arguments of \cite{sigma}, we can show the global well-posedness
in the Sobolev spaces $H^\sigma(\T)$, $\sigma\geq \beta/2$. 
We denote by $\Phi_t$, $t\in\R$ the associated flow and write $u(t)=\Phi_tu(0)$.

\subsection{Invariance of $\gamma_0$}
The measure $\gamma_0$ is special because we expect that it is invariant under $\Phi_t$ thanks to the $H^{\beta/2}$ conservation. 
The difficulty here is that for $\beta$ close to $1$ the flow defined in \cite{sigma} is by far not well defined on $H^s$, $s < \frac{\beta}{2}-\frac12$, which 
is the Sobolev regularity of $\g_{0}$-typical initial data. However, given $\beta >1$, we will be able to define locally in time a flow on $C^{\alpha}$
for $0 < \alpha < \frac{\beta}{2} - \frac{1}{2}$ (see Section \ref{sect:LWP}). Then, thanks to a well-known argument by Bourgain \cite{B94}, we will promote the local flow to a global one, $\g_{0}$-almost surely, using the invariance of the $\g_{0}$ measure. 
\begin{theorem}\label{TH:Invariance}
Let $\beta>1$. Equation \eqref{BBM-gamma} is globally well-posed for $\g_{0}$-almost all initial data. Moreover the measure~$\g_{0}$ is invariant under the resulting flow. 
\end{theorem}
Here we can exploit one standard characterisation of the support of $\gamma_0$, namely that $\bigcap _{0< \a<\frac{\beta-1}{2}} C^{\a}$ is a full $\gamma_0$-measure set. Thus combining Theorem~\ref{TH:Invariance} and the Poincar\'e recurrence theorem, we have for all $\beta >1$ recurrence of the solutions with respect to the $C^{\alpha}$ topology,  $ \alpha \in (0, \frac{\beta}{2} - \frac{1}{2})$, almost surely with respect to $\g_{0}$.  
\begin{corollary}\label{cor:recurrence}
Let $\beta>1$ and $\a\in(0, \frac{\beta-1}{2})$. For $\g_{0}$-almost all $u_0 \in C^{\alpha}$, there exists a diverging sequence of times $\{t_n\}_{n\in\N}$ such that
$$
\lim_{n \to \infty} \| \Phi_{t_n}u_0 - u_0\|_{C^{\alpha}} = 0 \, . 
$$
\end{corollary}
\subsection{Quasi-invariance of $\gamma_s$}
When $s>0$ we can still study the transport of $\g_{s}$ introducing suitable weights. In order to prove the quasi invariance of the Gaussian measure $\g_{s}$, we use a rigid cut-off on the $H^{\beta/2}$ norm, which is a conserved quantity (see \eqref{cons} below), and an additional exponential weight. We define the measure for $s > \beta /2$
\begin{equation}\label{Def:Rho}
\rho_{s}(du):=1_{\{ \| u \|_{H^{\beta/2}} \leq R \}} (u)\exp(-  \|u\|^{2r}_{H^s}) \g_{s}(du)\,, \quad r > 2. 
\end{equation}
We stress that we will not keep track of the dependence of $\rho_{s}$ on the parameters $R$, $r$ and $\b$ in the notations.
In the following theorem we cover the case $\beta\in(1,2]$. As explained in \cite{sigma}, the classical result of Ramer \cite{Ramer} applies for $\beta>2$.
\begin{theorem}\label{TH:quasi}
Let $\beta \in(1,2]$, $s > \frac\beta2$ such that $s + \beta/2 >3/2$. Let also $r>2$. 
The measures $\rho_{s}$ are quasi-invariant along the flow of \eqref{BBM-gamma}. The densities 
$f_{s}(t,u)$ of the transported measures are in~$L^p(\rho_{s})$ for all $t \in \R$ and~$p<\infty$. Moreover if $s>\frac32$ 
\begin{equation}\label{eq:Densities}  
f_s(t,u) := \exp \Big( -  \|  \Phi_t u\|^{2r}_{H^s} - \frac12 \| \Phi_t u\|^2_{ H^{s+\frac\beta2}} + \|   u\|^{2r}_{H^s} + \frac12 \|  u\|^2_{ H^{s+\frac\beta2}} \Big)\,. 
\end{equation}
\end{theorem}
As we shall see in Proposition~\ref {prop:convL1} and Proposition~\ref{Prop:density} below $f_s(t,u)$ can be obtained as the natural limit of the corresponding finite dimensional densities associated with the finite dimensional truncations of
 \eqref{BBM-gamma} (and this convergence can be used to define it). 
 
 It should be pointed out that in the expression for $f_s(t,u)$ there is an important cancellation in 
 \begin{equation}\label{cancel}
 -\frac12 \| \Phi_t u\|^2_{ H^{s+\frac\beta2}} + \frac12 \|  u\|^2_{ H^{s+\frac\beta2}}
 \end{equation}
 because each term \eqref{cancel} is not well-defined on the support of $\gamma_s$.
 Therefore, it is a part of the statement of Theorem~\ref{TH:quasi} that \eqref{cancel} is well-defined $\gamma_s$ almost surely. The same considerations are valid also for the NLS equation, see Theorem \ref{TH:quasiNLS} below. 
 
Note that the statement of  Theorem~\ref{TH:quasi} covers the full range $\beta \in (1, 2)$, $s \geq 1$. The restriction~$s > \beta /2$ is 
needed in order to take advantage of 
the exponential cut-off. Indeed for~$s \leq \beta /2$ the exponential cut-off gives no benefits, as the rigid cut-off on the~$H^{\beta/2}$ 
norm imposes already a stronger restriction (probably these cases can be dealt with more complicated probabilistic techniques, as in \cite{BBM2}). We believe the assumption $s>\frac32$ for the densities is merely technical and brings no special meaning. For $s\in(\frac\beta2,\frac32]$ indeed the proofs of Proposition \ref{lemma:ConvMeas} and Proposition \ref{lemma:unint} below get more involved and here we decided to put the focus elsewhere and not to burden the paper with technicalities, keeping the proofs at their easiest possible level.  

By varying the parameter $R$, we get the following corollary of Theorem~\ref{TH:quasi}. This result is new
for~$\beta \in (1, 4/3]$; for the case $\beta > 4/3$ we refer to \cite{sigma}.
\begin{corollary}
Under the assumptions of Theorem~\ref{TH:quasi}, the measure $\gamma_s$ is quasi-invariant under the flow $\Phi_t$.
\end{corollary}
As mentioned the restriction $\beta>1$ is natural for at least two reasons. The first one is that for $\beta\leq 1$ the measure $\gamma_0$ is no longer supported by classical functions and therefore the extension of Theorem~\ref{TH:Invariance} to $\beta\leq 1$ would require renormalisation arguments and this is a qualitatively different situation. 
The second reason is that it is not known whether \eqref{BBM-gamma} is globally well-posed for $\beta<1$. It is however known that \eqref{BBM-gamma} is globally well-posed for $\beta=1$ (see \cite{Mammeri}) and we plan to study the extension of  Theorem~\ref{TH:quasi}  to the case $\beta=1$ in a future work.

\subsection{The quintic NLS equation}

We prove similar results for the defocusing quintic NLS on $\T$:
\begin{equation}\label{NLSQuintic}
i \partial_t u + \partial_{x}^2 u = |u|^4 u, \quad u(0,x)=u_0(x).
\end{equation}

The $L^2$ norm of the solution $\|u\|_{L^2}$ and the energy
\begin{equation}
\mc E_1( u) = \frac12\|u\|^2_{H^1}+\frac16\|u\|^6_{L^6}\,
\end{equation}
are formally conserved by the flow. Therefore we can control the $H^1$ norm as
$$
\| u(t) \|^2_{H^{1}} \lesssim \|u(0)\|^2_{L^2} + \mc E_{1}(u(0))\,.
$$
This a priori estimate allows us to construct global solutions for all initial data in $H^{1}$, for which a local well-posedness 
theory is available by standard methods. 
Again, we denote the flow with $\Phi_t$, $t\in\R$ and write $u(t)=\Phi_tu(0)$.

The quintic NLS \eqref{NLSQuintic} is not an integrable system, which in particular means that we have no conserved quantities at our disposal to 
control higher order Sobolev norms. However, in \cite{PTV2}, \cite{PTV} a countable family of {\it modified energies} has been introduced. The derivative
along the flow of the modified energies is not zero, but it presents however some smoothing (see \eqref{1Smoothing}), which makes them 
still useful in order to 
control the growth in time of the Sobolev norms of the solutions.

\subsection{Quasi-invariance of $\g_{2k}$.} In order to study the transport property of the Gaussian measure $\g_{2k}$, we again study auxiliary weighted measures of the Gibbs type, constructed with the modified energies of \cite{PTV}. We also use a rigid cut-off on the conserved quantities introduced above, \ie mass and energy, along with an exponential cut-off on the $H^{2k-1}$ norm. We set
\begin{equation}\label{Def:RhoNLS}
\mu_{2k}(du):=1_{\{  \| u \|_{L^{2}}+\mc E_1(u) \leq R \}} (u)\exp(-R_{2k}( u)-\| u\|^{2r}_{H^{2k-1}})   \g_{2k}(du)\,,  
\end{equation}
where $R_k$ is a suitable functional (see Theorem \ref{th:PTV}). Also in this case the parameters $R$ and $r$ will be always omitted in the notations.

We prove the following statement. 
\begin{theorem}\label{TH:quasiNLS}
Let $k \geq 2$ be an integer. There exists $r(k)>0$ sufficiently large such that for all $r> r(k)$ the measures $\mu_{2k}$ are quasi-invariant along the flow of \eqref{NLSQuintic}. For all $t\in\R$, there exists 
$p = p(|t|) >1$ such that the densities 
$f_{2k}(t,u)$ of the transported measures are in~$L^p(\mu_{2k})$. Moreover   
\begin{equation}\label{eq:DensitiesNLS}  
f_{2k}(t,u) := \exp \Big( -  \|  \Phi_t u\|^{2r}_{H^{2k - 1}} - \mc E_{2k}(\Phi_t u)  + \| u \|^{2r}_{H^{2k - 1}} +  \mc E_{2k}( u) \Big)\,. 
\end{equation}
\end{theorem}

Again by varying the parameter $R$, we obtain another proof of the 
following result of \cite{PTV}.  
\begin{corollary}[\cite{PTV}]
The measure $\gamma_{2k}$, $k\geq 2$, is quasi-invariant under the flow $\Phi_t$.
\end{corollary}
\subsection{Comments on the method.}

The way a Gaussian measure transforms under the action of a given map is traditionally a very important topic in probability theory, starting from the classical works by Cameron-Martin \cite{CM} for constant shifts, by Girsanov \cite{girs} for non-anticipative maps (\ie adapted shifts) and by Ramer \cite{Ramer} for a certain class of anticipative maps (non-adapted shifts). Malliavin calculus brought further developments to the subject \cite{ABC1,ABC2, ust}, essentially establishing Jacobi formulas for Gaussian measures in functional spaces for more general classes of maps. Recently the paper 
by Debussche-Tsutsumi \cite{deb} (see also \cite{FS})  presented an approach to the density in many respects similar to ours, which however looks very much dependent on the dispersive nature of the equation under consideration, whereas dispersion plays no 
role in our method.
Indeed all the aforementioned classical results can be read in terms of the properties of the generator of the transformation, which is required to be of the Hilbert-Schmidt class (for a comprehensive survey see \cite{ust}).
This strong requirement is often violated in many cases of interest in the realm of dispersive PDEs, as it happens for either the flow maps studied in our work. Therefore the present method, elaborating on the previous paper \cite{sigma}, candidates to be a genuine extension of the Ramer theorem to flow maps with not necessarily Hilbert-Schmidt generators.

We will see that the presence of the exponential weight makes the analysis of the evolution of the 
measure much easier, since it improves the integrability properties of the measure
and, more importantly,
it helps us to control the time derivative of the evolution 
of the measure under the flow. 
On the other hand, once we compute this time derivative, we also need to control the 
contribution coming by the lack of conservation of the exponential weight, which is non trivial. 
In order to close the argument we need to balance this two competing effects in a suitable manner.

It is indeed worth mentioning that an a priori cut-off on a quantity which is not conserved 
is a highly non-trivial object, as the contribution to the measure evolution can exhibit a very singular behavior. For instance a (whatever smooth) compactly supported cut-off would be extremely hard to control.
However, for the exponential weight considered here, the additional contribution coming from the lack of conservation can be controlled by suitable energy estimates and by invoking the elementary inequality
$$
 \sup_{x\geq0} x^{a}e^{-\frac{x^{2r}}{p}}  \leq C p^{\frac{a}{2r}}
$$
for an appropriate positive constant $a\leq2r$, where $C$  is independent of $p\geq 1$ ($p$ is chosen large).

Of course the situation would be very different if we would work with conserved quantities,
in which case it would be substantially equivalent using a rigid (smooth) cut-off or an exponential one. This 
was for instance observed by Bourgain in  
\cite{BourgainBook} (see the remark on page 124) to construct the Gibbs measure for the one dimensional periodic NLS, which requires 
a cut-off on the (conserved) $L^2$ norm of the solution.

Using the exponential cut-off idea for the BBM equation, we extend to $\beta > 1$ the quasi invariance result of \cite{sigma}, concerning~$\beta > 4/3$. 
Notably the restriction~$\beta>1$ is the border line of the well-posedness theory. 

In  the context of the NLS equation, the exponential weights allow us to construct quasi-invariant measures which enjoy few more properties than the ones in \cite{PTV}.
More precisely, besides the quasi invariance, we can provide an explicit formula for the density 
(see \eqref{eq:DensitiesNLS}) and prove $L^p$ properties of the Radom-Nikodym derivatives.

It is possible to cover also the case $k=1$ with the techniques developed in this paper. However, since the time derivative of the corresponding quasi-energy $\mathcal{E}_2$ satisfies a slightly weaker estimate than~\eqref{1Smoothing} (see \cite[Theorem 1.4]{PTV}), a more technical argument would be needed, which again we preferred to skip.

\subsection{Organisation of the paper.}
The rest of the paper is organised as follows. The deterministic estimates are presented in the next three sections. In Section \ref{sect:LWP} we establish a local well-posedness result for the BBM equation in low-regularity H\"older's spaces. In Section \ref{sect:control} we prove two useful deterministic estimates in Sobolev spaces for the BBM equation.  In  Section~\ref{Sec:DeterministicNLS} we present analogue deterministic estimates in Sobolev spaces for the NLS equation as long as what we need on the modified energies introduced in \cite{PTV}. 
In Section \ref{sect:invarianza} we construct the flow for the BBM equation for $\gamma_0$-almost all initial data and prove Theorem \ref{TH:Invariance}.
In Section \ref{sect:quasi} we prove the quasi-invariance part of Theorem \ref{TH:quasi}, and in Section \ref{sec:density} we complete it with the density for $s>\frac32$. 
In Section \ref{sect:quasiNLS} we prove the quasi-invariance part of Theorem \ref{TH:quasiNLS} and then in Section \ref{DensityNLS} we find the 
explicit density, completing the proof of Theorem \ref{TH:quasiNLS}.

\subsection{Notations.} A centred ball of radius $R$ in the $H^s$ topology is denoted by $B^s(R)$.
 We drop the superscript for $s=0$ (balls of $L^2$). 
 We will sometimes write $C^{\beta/2}$ when $s=\beta/2$.
 In Section \ref{sect:LWP} and  Section~\ref{sect:invarianza} we deal also with a centred ball of radius $R$ in the $C^\a$ topology, denoted by $B^\a(R)$. In general a Greek letter superscript always refers to ball in H\"older's spaces, with exception of 
$C^{\beta/2}$ that is used for balls in Sobolev spaces (this exception is due to the fact that $\beta$ is special for us,
being the dispersion parameter in the BBM equation \eqref{BBM-gamma}). 
For two quantities $X$ and $Y$, we write $X\lesssim Y$ if there is a uniform constant $c>0$ such that $X\leq cY$ for every choice of $X$, $Y$. We write $X\simeq Y$ if $Y\lesssim X\lesssim Y$. We underscore the dependency of $c$ on the additional parameter $a$ writing $X\lesssim_a Y$. $C,c$ always denote constants that often vary from line to line within a calculation. We denote by $P_N$ the orthogonal projection defined by
$$
P_N(u)=\sum_{|n|\leq N}\hat{u}(n)e^{inx}\,
$$
(recall that $\hat{u}(n)$ is the $n$-th Fourier coefficient of $u\in L^2$). 
By convention, $P_{\infty}={\rm Id}$.
Also, we denote the Littlewood-Paley projector by $\D_0:=P_{1}$, 
$\D_j:=P_{2^{j}}-P_{2^{j-1}}$, $j \in \N$. We use the standard notation $[A,B] := AB - BA$ to denote the commutator of the operators $A, B$. We will denote the flow of either the BBM and NLS equation by $\Phi_t$ (and the truncated flow by $\Phi_t^N$), where the difference will always be clear by the context.

\subsection{Acknowledgements.} This paper benefited by the comments of an anonymous referee, who is gratefully acknowledged.
R. Luc\`a is supported by the Basque Government under program BCAM- BERC 2022-2025 and
by the Spanish Ministry of Science, Innovation and Universities under the BCAM Severo Ochoa 
accreditation SEV-2017-0718 and by the projects PGC2018-094528-B-I00 and PID2021-123034NB-I00.
N. Tzvetkov is supported by ANR grant ODA (ANR-18-CE40-0020-01).

%%%%%%%%%%%%%%%%%%%%%%%%%%%%%%%%%%%%%%%%%%%%%%%%%%%%%%%%%%%%%%%%%%%%%%%%%%%%%%%%%%%%%%%%%%%%%%%%%%%%%%%%%%%%%%%%%%%%%%%%%%%%%%%%%%%%%%%%%%%%%%%%%%%%%%%%%%%%%%%%%%%%%%%%%%%%%%%%%%%%%%%

\section{Local well-posedness in H\"older's spaces}\label{sect:LWP}
Let $\alpha \in (0,1)$. We define as usual
\begin{equation}\label{HolderNorm}
\| f \|_{C^{\alpha}} := \| f \|_{L^{\infty}} +  \sup_{x\neq y }   \frac{|f(x) - f(y)|}{|x-y|^{\alpha}}. 
\end{equation}
Let 
$$
L_{\beta} := \frac{ \partial_x }{1 + |D_x|^\beta}  .
$$
The following statement follows by a basic application of the Littlewood-Paley theory (see e.g. \cite{BCD}). 
\begin{lemma}. 
Let $N \in \mathbb{N} \cup \{0\}$, 
$\beta >1$, $\alpha \in (0,1)$ and
$$
\varepsilon \in [0, \min(\alpha, \beta -1))\,.
$$ 
We have
\begin{equation}\label{CAlphaBoundedness}
\| L_{\beta} f  \|_{C^{\alpha}} \lesssim \|  f  \|_{C^{\alpha-\varepsilon}}.
\end{equation}
As a consequence 
\begin{equation}\label{CAlphaBoundedness2}
\| L_{\beta} (fg) \|_{C^{\alpha}} \lesssim \| f  \|_{C^{\alpha - \varepsilon}} \| g \|_{L^{\infty}}  
+ \| f \|_{C^{\alpha- \varepsilon}} \| g \|_{L^{\infty}}  \, .
\end{equation}
\end{lemma}
Indeed, one can easily prove that the statement holds when $\varepsilon = 0$, for all $\beta' >1$. Then 
the self-improved estimate \eqref{CAlphaBoundedness} follows taking $1< \beta' < \beta$ 
and writing $L_{\beta} = \frac{L_{\beta}}{L_{\beta'}} L_{\beta'}$, 
since the operator $\frac{L_{\beta}}{L_{\beta'}}$ gives a $\varepsilon=\beta-\beta'$ regularizsation.

Note that \eqref{CAlphaBoundedness} implies, by Taylor expansion of the exponential 
\begin{equation}\label{ExpBound}
\| e^{-t L_{\beta}} f \|_{C^{\alpha}} \leq e^{C |t|} \| f \|_{C^{\alpha}}.
\end{equation}
Now we study the truncated equation
\begin{equation}\label{BBM-gamma-N}
\partial_t u+\partial_t|D_x|^\beta u+\partial_x u+\partial_xP_N((P_Nu)^2)=0, \quad u(0,x)=u_0(x)\,.
\end{equation}
Thanks to \cite{sigma} we can show  that \eqref{BBM-gamma-N} is globally well-posed in the Sobolev spaces $H^\sigma(\T)$, $\sigma\geq \beta/2$. 
We denote by $\Phi^N_t$ the associated flow. Recall that $\Phi^{N=\infty}_t=\Phi_t$. 
In the next proposition,  we shall show that \eqref{BBM-gamma-N} is locally well-posed in~$C^{\alpha}$, for all~$\alpha \in (0,1)$ and thus we will extend locally in time $\Phi^N_t$ to H\"older's spaces. 
\begin{proposition}\label{LocWPCalpha}
Let $\alpha \in (0,1)$. There exists a sufficiently small constant $c >0$ (independent of $N$) such that the following holds.
Let $K>0$ and $u(0)$ such that $\| u(0)\|_{C^{\alpha}} \leq K$. There is a unique solution of \eqref{BBM-gamma-N}
$u(t) \in C([0,T]; C^{\alpha})$, $T:=\frac c{1+K}$. Moroever
$$
\sup_{t \in [0,T]}\| \Phi_tu(0) \|_{C^{\alpha}} \leq  2K\,.
$$   
The same holds also for the equation (\ref{BBM-gamma}) (\ie the case $N=\infty$). 
\end{proposition}
\begin{proof}
We rewrite \eqref{BBM-gamma-N} as
$$
\partial_t u +  L_{\beta}    u  + L_{\beta}   P_N ((P_Nu)^2) = 0   \,,
$$
and its Duhamel formulation is
\begin{equation} \label{DuhamForm}
u(t) = e^{-t L_{\beta}} u(0) - \int_{0}^{t} e^{-(t-s) L_{\beta}}   L_{\beta}    P_N ((P_Nu)^2)(s) \, ds.  
\end{equation}
Using the estimates \eqref{CAlphaBoundedness}-\eqref{CAlphaBoundedness2} a solution to \eqref{DuhamForm} can constructed as the fixed point of the map
\be\label{eq:mapF}
F_{u(0)} : u \mapsto e^{-t L_{\beta}} u(0) - \int_{0}^{t} e^{-(t-s) L_{\beta}}  L_{\beta} P_N ((P_Nu)^2) \, ds \, . 
\ee
Indeed, we take $u \in C([0,T]; C^{\alpha})$ 
such that 
$$
\sup_{t \in [0,T]} \| u(t) \|_{C^{\alpha}} \leq 2K.
$$ 
Using \eqref{CAlphaBoundedness2}, \eqref{ExpBound}, the  
bound 
\begin{equation}\label{ExpBound2}
\| P_N u \|_{C^{\alpha-\varepsilon}}  \leq C \| P_N u \|_{C^{\alpha}}, 
\qquad (\mbox{$C$ independent on $N \in \N \cup\{\infty\}$})  
\end{equation}
 and Minkowski integral 
inequality, we can estimate (recall that $T = \frac{c}{1+K}$)
$$
\| F_{u(0)}(u) \|_{L^{\infty}([0,T]; C^{\alpha})} \leq  e^{CT} \|  u(0) \|_{C^{\alpha}} + C T e^{CT} \| (P_N u)^2\|_{L^{\infty}([0,T]; C^{\alpha})} \leq \frac{3}{2} K + C T K^2 \leq 2K, 
$$
provided $KT$ (and therefore $c$) is small enough. Thus $F_{u(0)}$ maps any centred ball of radius $2K$ in the $C([0,T]; C^{\alpha})$ topology into itself. To show that it contracts the distances, we note
$$
(P_N u)^2 - (P_N v)^2 = (P_N u + P_N v ) (P_N u - P_N v)
$$
so that given $u, v$ with 
$$
\sup_{t\in[0,T]}\|u(t)\|_{C^\a}, \sup_{t\in[0,T]}\|v(t)\|_{C^\a}\leq 2K
$$
again by \eqref{CAlphaBoundedness2}-\eqref{ExpBound}-\eqref{ExpBound2} and 
Minkowski integral 
inequality we get
\begin{align}\nonumber
\| F_{u(0)}(u) - F_{u(0)}(v) \|_{L^{\infty}([0,T]; C^{\alpha})} 
& \leq    C T e^{CT} 
\|  P_N u + P_N v  \|_{L^{\infty}([0,T]; C^{\alpha})} \|  P_N u - P_N v  \|_{L^{\infty}([0,T]; C^{\alpha})} 
\\ \nonumber
&
\leq C T K  \|  P_N u - P_N v  \|_{L^{\infty}([0,T]; C^{\alpha})}  \leq \frac12  \|  u -  v  \|_{L^{\infty}([0,T]; C^{\alpha})} \,   ,
\end{align}
provided $c>0$ is small enough. This allows us to use the
contraction theorem to prove the existence of the solution. 
The uniqueness statement can be obtained by similar arguments. 
\end{proof}
We also need some stability results for the flow map $\Phi_t^N$ associated to the local solutions from Proposition~\ref{LocWPCalpha}.
Recall that we write simply $\Phi_t$ instead of $\Phi_t^\infty$ and
\be\label{eq:cialfaball}
B^{\alpha}(K) := \{ f \in C^{\alpha} : \| f \|_{C^{\alpha}} \leq K \}.
\ee
The proofs of the next two lemmas are straightforwardly adapted from the one of Proposition \ref{LocWPCalpha} so we omit them.
%%%
\begin{lemma}\label{Lemma:Stability}
Let $0 < \alpha' < \alpha < 1$. There exists a sufficiently small constant $c >0$ and such that the following holds.
Let $K>0$ and $T = \frac{c}{1+K}$. Then for all $N\in\N$
\begin{equation}\label{FlowMapBad}
\sup_{ |t| \leq T }   \sup_{u(0), v(0)  \in B^{\a}(K)}
\|  \Phi_t u(0) -  \Phi_t^N u(0)  \|_{C^{\alpha'}} 
\lesssim  K  N^{\alpha' - \alpha}  \,.
\end{equation} %  
\end{lemma}
\begin{lemma}\label{Lemma:StabilityBis}
Let $\alpha \in (0,1)$. There exists a sufficiently small constant $c >0$ and such that the following holds.
Let $K>0$ and $T = \frac{c}{1+K}$. Then
\begin{equation}\label{FlowMapBadBis}
\sup_{ |t| \leq T }  \,  \sup_{u(0), v(0)  \in B^{\a}(K)}
\|  \Phi_t u(0) -  \Phi_t v(0)  \|_{C^{\a}} 
\leq 4 \| u(0) - v(0)  \|_{C^{\a}} 
\, .
\end{equation}
\end{lemma}
We conclude the section proving an approximation result, which tells us how to construct the actual flow $\Phi_t$ on a time
interval $t \in [0,T]$ (now $T$ may be large) on which the approximated flow~$\Phi^N_t$ is suitably bounded for some sufficiently large value of $N$.  
Notice that in Proposition~\ref{ApproxThm} it is in principle not necessary to assume that \eqref{GrowthAssumption} holds for all sufficiently large $N$, but 
only for given value of $N$, say $N=\overline{N}$, large enough. On the other hand, since this $\overline N$ depends on the input parameters (in particular $K$)
in a complicated way, in order to verify \eqref{GrowthAssumption} in practical situations, we will rather need to prove its validity for all $N$ sufficiently large.  

\begin{proposition}\label{ApproxThm}
Let $0 < \upsilon' < \upsilon < 1$, $K, T >0$ and $\varepsilon \in (0,K)$. Let $A \subset B^{\upsilon}(K)$ (see \ref{eq:cialfaball}). 
There exists $N$ sufficiently large (depending on $\upsilon, \upsilon',\varepsilon,K, T$) such that the following holds. If
\begin{equation}\label{GrowthAssumption}
\sup_{t \in [0, T]} \sup_{u(0) \in A} \| \Phi^N_{t} u (0) \|_{C^{\upsilon}} \leq K , 
\end{equation}
then
the flow $\Phi_{t} u(0)$ is well defined on $t \in [0,T]$ for all $u(0) \in A$. Moreover
\begin{equation}\label{approxProp}
\sup_{t \in [0, T]} \| \Phi_{t} u(0) - \Phi^N_{t} u (0) \|_{C^{\upsilon'}} \leq \varepsilon, 
\quad \forall u(0) \in A\,.
\end{equation}
\end{proposition}
%%%
\begin{proof}
Let $J$ be the smallest integer such that $J \frac{c}{2K+1}  \geq T$, where $c$ is given by
Proposition \ref{LocWPCalpha} (possibly taking the smallest of such $c$). 
Clearly $$T (2K+1) c^{-1} \leq J < T (2K+1) c^{-1} + 1.$$  
We partition the interval $[0,T]$ into $J-1$ intervals of length $\frac{c}{2K+1}$ and a last, possibly smaller interval. 
The goal is to compare the approximated flow $\Phi^N_{t} u(0)$ 
and $\Phi_{t} u(0)$ (which exists only locally) on these small intervals and gluing local solutions.
Let  
$$\upsilon'  < b_{J} < \ldots <  b_2 < b_1 < b_0=\upsilon\,.$$
We proceed by induction over $j = 0, \dots, J$. Assuming 
that  $\Phi_{t} u(0)$ is well defined on $\big[0, (j+1)\frac{c}{2K+1}\big]$  and that
\begin{equation}\label{Growth1}
\sup_{t \in \big[0, j\frac{c}{2K+1}\big]} 
\|  \Phi_{t} u(0) -  \Phi_{t}^N u(0)   \|_{C^{b_{j}}} \leq  N^{-\kappa_j} \, ,
\end{equation}
for some $\kappa_j >0$, we will show that 
$\Phi_{t} u(0)$ is well defined on $\big[0, (j+2)\frac{c}{2K+1}\big]$  and that
\begin{equation}\label{Growth2}
\sup_{t \in \big[0, (j+1)\frac{c}{2K+1}\big]} 
\|  \Phi_{t} u(0) -  \Phi_{t}^N u(0) \|_{C^{b_{j+1}}} \leq   N^{-\kappa_{j+1}} \, ,
\end{equation}
for a suitable $\k_{j+1} >0$, 
provided $N$ is sufficiently large. In particular, we take $N$ so large in such a way that we also have 
$N^{-\kappa_{j+1}} < \varepsilon$. 
Using the induction procedure up to~$j = J$, the statement would then follows.

The induction base $j=0$ is covered by Proposition \ref{LocWPCalpha} and by 
the fact that $A \subset B^{\upsilon}(K)$.

Regarding the induction step, we only need to prove \eqref{Growth2}. Assuming indeed that 
 \eqref{Growth2} is proved, using the assumption~\eqref{GrowthAssumption} and triangle 
inequality, we have 
$$
\sup_{t \in \left[0, (j+1) \frac{c}{2K+1}\right]} 
\|  \Phi_{t} u(0) \|_{C^{b_{j+1}}} \leq  K +  N^{-\kappa_{j+1}} < K + \varepsilon < 2K.
$$
So we can use 
Proposition \ref{LocWPCalpha} (with $2K$ in place of $K$) to show that
$\Phi_{t} u(0)$ is well defined on $\big[0, (j+2)\frac{c}{2K+1}\big]$.

So, it remains to show \eqref{Growth2}. If the $\sup$ in \eqref{Growth2} is attained
for~$t \in \big[0, j \frac{c}{2K+1}\big]$, then \eqref{Growth2} follows by \eqref{Growth1} simply taking $\kappa_{j+1} = \kappa_{j}$.
On the other hand, if the $\sup$ is attained for $t \in \big[ j \frac{c}{2K+1} , (j+1) \frac{c}{2K+1} \big]$, using the group property of the flow, we need to prove 
\begin{equation}\label{Growth3}
\sup_{t \in \big[0,  \frac{c}{2K+1} \big]} 
\|  \Phi_{t}  \Phi_{ j\frac{c}{2K+1} } u(0) -  \Phi_{t}^N   \Phi_{ j\frac{c}{2K+1}}^N u(0) \|_{C^{b_{j+1}}} 
\leq   N^{-\kappa_{j+1}} \, .
\end{equation}
To do so we decompose 
\begin{align}\nonumber
\|  \Phi_{t}  \Phi_{j\frac{c}{2K+1} } u(0) & -  \Phi_{t}^N  \Phi_{j\frac{c}{2K+1}}^N u(0) \|_{C^{b_{j+1}}}
\\ \label{Term1}
& 
\leq \|  \Phi_{t}  \Phi_{j\frac{c}{2K+1}} u(0) -  \Phi_{t}  \Phi_{j\frac{c}{2K+1}}^N u(0) \|_{C^{b_{j+1}}}
\\ \label{Term2}
& + \|  \Phi_{t}  \Phi_{j\frac{c}{2K+1}}^N u(0)-  \Phi_{t}^N  \Phi_{j\frac{c}{2K+1}}^N u(0) \|_{C^{b_{j+1}}}
\end{align}
and we will handle these two terms separately.

To prove \eqref{Growth3} for the term \eqref{Term1}
we first note that by the induction assumption \eqref{Growth1} and by the assumption \eqref{GrowthAssumption} we have
for $N$ large enough
$$ 
\|  \Phi_{j\frac{c}{2K+1}} u(0) \|_{C^{b_j}} \leq K + N^{-\kappa_{j}} < K+\varepsilon.
$$
Using this fact and the assumption \eqref{GrowthAssumption}, we are allowed to apply the stability estimate \eqref{FlowMapBadBis}  
with $\a= b_j$, which leads to us
\be\label{eq:rhs-of}
\sup_{t \in [0, \frac{c}{2K+1} ]}\|  \Phi_{t}  \Phi_{j\frac{c}{2K+1}} u(0) -  \Phi_{t}  \Phi_{j\frac{c}{2K+1}}^N u(0) \|_{C^{b_{j+1}}}\lesssim \|\Phi_{j\frac{c}{2K+1}} u(0) - \Phi_{j\frac{c}{2K+1}}^N u(0) \|_{C^{b_{j+1}}} 
\lesssim N^{-\kappa_j} 
\,,
\ee
where in the last inequality we used the induction assumption \eqref{Growth1}. Thus, letting we arrive to
\begin{equation}\label{InPart}
\mbox{r.h.s. of }\eqref{eq:rhs-of}\leq \frac12 N^{- \k_{j+1}} < \frac{\varepsilon}{2},
\end{equation}
for
$\kappa_j \geq \k_{j+1}  > 0$,  
provided that $N$ is sufficiently large. 

To prove \eqref{Growth3} for the term \eqref{Term2}
we use the stability estimate \eqref{FlowMapBad} with $\a=b_j, \a'=b_{j+1}$ and initial datum 
$\Phi_{j\frac{c}{2K+1}}^N u(0)$,
that is allowed recalling the assumption \eqref{GrowthAssumption}. Thus for 
$b_j - b_{j+1} \geq \kappa_{j+1} \geq 0$, we arrive to 
$$
\sup_{t \in [0, \frac{c}{2K+1}]} 
\|  \Phi_{t}  \Phi_{j\frac{c}{2K+1}}^N u(0) 
-  \Phi_{t}^N  \Phi_{j\frac{c}{2K+1}}^N u(0) \|_{C^{b_{j+1}}}
\lesssim K N^{b_{j+1}-b_j} < \frac12 N^{-\kappa_{j+1}}.
$$
provided that $N$ is sufficiently large (in particular such that \eqref{InPart} holds). Choosing 
$\kappa_{j+1} = \max (\kappa_{j}, b_j - b_{j+1})$, this concludes the proof of Proposition~\ref{ApproxThm}.
\end{proof}

%%%%%%%%%%%%%%%%%%%%%%%%%%%%%%%%%%%%%%%%%%%%%%%%%%%%%%%%%%%%%%%%%%%%%%%%%%%%%%%%%%%%%%%%%%%%%%%%%%%%%%%%%%%%%%%%%%%%%%%%%%
%%%%%%%%%%%%%%%%%%%%%%%%%%%%%%%%%%%%%%%%%%%%%%%%%%%%%%%%%%%%%
\section{Control of Sobolev norms for the BBM equation}\label{sect:control}
Using the conservation of the $H^{\beta/2}$ norm we can prove sub-quadratic growth for higher order Sobolev norms. 
More precisely, paring in $L^2$ equation \eqref{BBM-gamma-N} with $P_N u$, we obtain that any solution $u$ satisfies 
\begin{equation}\label{cons}
\|   P_N u(t) \|_{H^{\beta/2}} = \|   P_N u(0) \|_{H^{\beta/2}} \,;
\end{equation}
we refer to \cite[Lemma 2.4]{sigma} for details.

\begin{proposition}\label{prop:growth-Hs-norm}
Let $\beta>1$, $\sigma>\beta/2$ and $\alpha>0$ such that
\begin{equation}
\label{restrict}
1+\alpha<\beta,\quad \sigma-\alpha>\beta/2\,.
\end{equation}
Let $R >0$. Then any solution of \eqref{BBM-gamma-N}  with initial datum $u(0)$ such that $\| u(0) \|_{H^{\beta/2}}  \leq R$ 
 satisfies for all $t \in \R$: 
\begin{equation}\label{DI1}
\left| \frac{d}{dt}\|P_N u(t)\|_{H^\sigma}^2 \right| \lesssim R^{1+\theta} \|P_N u(t)\|_{H^\sigma}^{2-\theta}\,,\quad\theta:=\frac{2\alpha}{2\sigma - \beta}\in(0,1)\,.
\end{equation}
\end{proposition}

\begin{remark}
From \eqref{DI1} one can deduce polynomial growth for the higher Sobolev norms of the solution. More precisely for all $u(0)$ such 
that $\| u(0) \|_{H^{\beta/2}}  \leq R$ one has
\begin{equation}\label{PolynomialGrowth}
\|P_N u(t)\|_{H^\sigma}  \leq C_R \, |t|^{1/\theta} \|P_N u(0)\|_{H^\sigma} \,.
\end{equation}
Better polynomial estimates for large $s$ can be obtain by the smoothing inequality \eqref{H2Bis}, 
however we will not pursue this matter any further.
\end{remark}

\begin{proof}
Let 
$
\Lambda(\beta) := \sqrt{1 + |D_x|^{2 \beta}}.
$
From \eqref{BBM-gamma-N} we have 
\begin{equation}\label{En1Preq}
\partial_t   \Lambda(\sigma) P_N u(t) = -  \frac{ \Lambda(\sigma) }{1 + |D_x|^\beta } \left(  \partial_x P_N u(t) +  \partial_x P_N ((P_Nu(t))^2) \right).
\end{equation}
Since
\begin{align}\nonumber
\int  \Big(  \frac{ \Lambda(\sigma) }{1 + |D_x|^\beta }  \Big(  \partial_x P_N u(t)  \Big)  \Big)
&  \Lambda(\sigma)  P_N u(t) =
\\ \nonumber
= \int  \Big(  \Big(  \partial_x \frac{ \Lambda(\sigma) }{\Lambda(\beta/2)}   P_N u(t)  \Big)  \Big) 
&  \frac{\Lambda(\sigma) }{\Lambda(\beta/2)}  P_N u(t)  
=
\frac{1}{2}  \int \partial_x  \Big(  \Big( \frac{ \Lambda(\sigma) }{\Lambda(\beta/2) } P_N u(t) \Big)^2 \Big)  =0,
\end{align}
pairing \eqref{En1Preq} in $L^{2}$ with $\Lambda(\sigma) P_N u$ gives us
\begin{align}\label{PivotB}
\frac{d}{dt}\|P_N u(t) \|_{H^\sigma}^2 
& = - \frac{1}{\pi} \int \Big(\Lambda(\sigma) P_N u(t) \Big)\,\Big( \frac{\Lambda(\sigma)}{1+|D_x|^\beta}  \partial_x ( (P_Nu(t) )^2)\Big);
\end{align}
note that on the right hand side we can write $\partial_x ( (P_Nu )^2)$ in place of $P_N \partial_x ( (P_Nu )^2)$ by orthogonality.
We rewrite the identity \eqref{PivotB} as
\begin{align}\label{ThisGives}
\frac{d}{dt}\|P_N u(t) \|_{H^\sigma}^2 
& = - \frac{1}{\pi} \int \Big(\frac{\Lambda(\sigma)}{\Lambda(\alpha)}  P_N u(t) \Big)\,
\Big( \frac{\Lambda(\alpha) \partial_x }{1+|D_x|^\beta}   \Lambda(\sigma) ( (P_N u(t) )^2)\Big).
\end{align}
Since $\alpha + 1 < \beta$ we have that $ \frac{\Lambda(\alpha) \partial_x}{1+|D_x|^\beta}$ is bounded on $L^2$, so that 
\eqref{ThisGives} gives
\begin{equation}\label{P1}
\left| \frac{d}{dt}\|P_N u(t) \|_{H^\sigma}^2 \right| \lesssim \|P_N u(t) \|_{H^{\sigma-\alpha}}\|(P_N u(t) )^2\|_{H^{\sigma}}\,.
\end{equation}
By \eqref{cons} and \eqref{restrict} there is $\theta \in (0,1)$ such that
\begin{equation}\label{Def:Theta}
\sigma - \alpha = \sigma (1-\theta)  + \frac{\beta}{2} \theta \,.
\end{equation}
Thus we can interpolate
\begin{equation}\label{P2}
\|P_N u(t) \|_{H^{\sigma-\alpha}}\leq \|P_N u(t) \|_{H^\sigma}^{1-\theta}\|P_N u(t)\|_{H^{\beta/2}}^\theta \leq R^\theta \|P_N u(t)\|_{H^\sigma}^{1-\theta}\,,
\end{equation}
where in the last estimate we used 
$\| P_N u (t) \|_{H^{\beta/2}} = \| P_N u (0) \|_{H^{\beta/2}} \leq R$ (see \eqref{cons}).
Finally
\begin{equation}\label{P3}
\|(P_Nu(t))^2\|_{H^{\sigma}}\lesssim\|P_Nu(t)\|_{H^{\sigma}}\|P_Nu(t)\|_{L^{\infty}}
\lesssim\|P_N u(t)\|_{H^{\sigma}}\|P_Nu(t)\|_{H^{\beta/2}}\leq R\|P_N u(t)\|_{H^{\sigma}},
\end{equation}
where we used again $\| P_N u (t) \|_{H^{\beta/2}} \leq R$ in the last bound.
Plugging \eqref{P2}-\eqref{P3} into \eqref{P1} gives the desired inequality \eqref{DI1}. 
\end{proof}
We also have $\frac{\beta}{2}$-smoothing for the time derivative of the $H^{s+\beta/2}$ norm.
\begin{proposition}\label{prop:energy}
Let $\beta > 1$ and $s > 1/2$. 
The solutions of \eqref{BBM-gamma-N} satisfy for all~$t \in \R$: 
\begin{equation}\label{GammaSmoothing}
\left| \frac{d}{dt}\|P_Nu(t)\|_{H^{s+\beta/2}}^2 \right| \lesssim  \|P_N u(t)\|_{H^s}^3 + \|P_N u(t)\|_{H^s}^2 \|\partial_xP_N u(u)\|_{L^{\infty}} \,.
\end{equation}
\end{proposition}
\begin{proof}
As in the previous proof we set
$
\Lambda(\beta) := \sqrt{1 + |D_x|^{2 \beta}}
$
and write 
\begin{align}\nonumber
\frac{d}{dt}\|P_N u(t) \|_{H^{s + \beta/2}}^2 
& = - \frac{1}{\pi} \int \Big(\Lambda(s + \beta/2) P_N u(t) \Big)\,
\Big(  \frac{\Lambda(s + \beta/2)}{1+|D_x|^\beta}  \partial_x ( (P_N u(t) )^2)\Big)\,.
\end{align}
This is nothing but \eqref{PivotB} with $\sigma = s + \beta/2$.
We rewrite this identity as 
\begin{align}\nonumber
\frac{d}{dt}\|P_N u(t) \|_{H^{s + \beta/2}}^2 
& = - \frac{1}{\pi} \int \Big(   M_1(D_x) \Lambda(s) P_N u(t) \Big)\,\Big( \Lambda(s)  \partial_x \big( (P_N u(t)  )^2 \big) \Big)\,.
\end{align}
with
$$
M_1(D_x) := \frac{1 + |D_x|^{2s + \beta} }{(1+|D_x|^\beta)(1 + |D_x|^{2s})}\,.
$$
Thus, writing
$$
M_1(D_x)  = 1 + M_2(D_x)\,,
$$
where 
$$
M_2(D_x) := - \frac{|D_x|^{2s} +  |D_x|^{\beta}}{(1+|D_x|^\beta)(1 + |D_x|^{2s})}\,,
$$
we decompose
$$
\frac{d}{dt}\|P_N u(t)\|_{H^{s+\beta/2}}^2=I_1+I_2,
$$
where 
\begin{align}\nonumber
I_1 & := -\frac{1}{\pi}\int (\Lambda(s) P_N u(t))\,\Lambda(s)
  \partial_x \big( (P_Nu(t))^2 \big)
\\ \nonumber
&
 =  -\frac{2}{\pi} \int (\Lambda(s)P_N u(t))\,\Lambda(s)
 \big((\partial_x(P_Nu(t))) P_Nu(t) \big)
\end{align}
and
\begin{align}\nonumber
I_2 
& :=\frac{1}{\pi}\int (M_2(D_x) \Lambda(s) P_N u(t))\,
 \Lambda(s)
 \big(\partial_x(P_Nu(t))^2 \big) 
 \\ \nonumber
 & 
 = 
 \frac{1}{\pi}\int (\Lambda(s) P_N u(t))\,
 M_2(D_x) \partial_x \Lambda(s)
 \big((P_Nu(t))^2 \big)\,.
\end{align}

 Since $\beta >1$ and $2s >1$, we immediately see that that $M_2(D_x) \partial_x $ is bounded on $L^2$. Thus we can estimate 
\begin{equation}\label{IneffLInf}
|I_2|\lesssim \|P_N u(t)\|_{H^s}\|(P_N u(t))^2\|_{H^s} \lesssim 
\|P_N u(t)\|_{H^s}^2 \| P_N u(t)\|_{L^{\infty}}
\lesssim
\|P_N u(t)\|_{H^s}^3 \,,
\end{equation}
where we used $s>1/2$ in the second inequality. On the other hand, 
using 
\begin{align}\nonumber
\Lambda(s)&  \big( (\partial_x P_N u(t)) (P_N u(t)) \big) 
\\ \nonumber
&= \big( \Lambda(s) \partial_x P_N u(t) \big) P_N u(t) + 
\big[ \Lambda(s), P_N u(t) \big] \partial_x P_N u(t)
\end{align}
we can rewrite $I_1$ as
\begin{equation}\label{I_1Decomp}
I_1 = \widetilde{I_1} 
-\frac{2}{\pi} \int (\Lambda(s) P_N u(t))\,
 \big[ \Lambda(s), P_N u(t)  \big] \partial_x P_N u(t),
\end{equation}
where
\begin{equation}\label{Def:TildeI}
\widetilde{I_1} := 
 -\frac{2}{\pi}\int (\Lambda(s) P_N u(t))\,( \Lambda(s)
\partial_x P_N u(t)) (P_N u(t)) .
\end{equation}
Integrating by parts in \eqref{Def:TildeI} allows us to rewrite
\begin{equation}\label{Def:TildeIBetter}
\widetilde{I_1} := 
  \frac{2}{\pi} \int \big| \Lambda(s) P_N u(t) \big|^2 \partial_x P_Nu(t) \,,
\end{equation}
so that 
\begin{equation*}
| \widetilde{I_1} | \leq \| P_N u(t) \|_{H^s}^2 \| \partial_x P_N u(t) \|_{L^{\infty}}\,. 
\end{equation*}

Thus it remains to bound the second contribution in \eqref{I_1Decomp}. To do so we use the following commutator estimate
 \cite{KenigPilod}, valid for $f$ periodic:
$$
\big\| \big[  \Lambda(s) , f \big] g  \big\|_{L^2} \lesssim ( \|  f \|_{L^{\infty}} + \| \partial_x f \|_{L^{\infty}} ) \| g \|_{H^{s-1}} + \| f \|_{H^s} \| g \|_{L^{\infty}} \,.
$$
This gives us
\begin{align*}\nonumber
\big\| \big[ \Lambda(s) ,  P_N u(t) \big] \partial_x P_N u(t)  \big\|_{L^{2}}
& \lesssim
( \|  P_N u(t) \|_{L^{\infty}} + \| \partial_x P_N u(t) \|_{L^{\infty}} ) \| \partial_x P_N u(t) \|_{H^{s-1}} 
\\ \nonumber&
+  \|  P_N u(t) \|_{H^s} \| \partial_x P_N u(t) \|_{L^{\infty}}
\\ \nn
& \lesssim \| P_N u(t) \|_{H^{s}}^2 + \| P_N u(t) \|_{H^{s}} \| \partial_x P_N u(t) \|_{L^{\infty}} , 
\end{align*}
whence
\begin{align}
& \left| \int (\Lambda(s) P_N u(t))\,
 \big[ \Lambda(s), P_N u(t)  \big] \partial_x P_N u(t) \right|
\\ \nn
& \qquad \lesssim 
 \| P_N u(t) \|_{H^{s}}^3 +  \| P_N u(t) \|_{H^{s}}^2 \| \partial_x P_N u(t) \|_{L^{\infty}}  \, ,
\end{align}
that concludes the proof.
\end{proof}
%%%%
We conclude the section with a further deterministic bounds, whose proof is very close to that of Proposition \ref{prop:energy}. We define, as usual 
$$
\| f \|_{W^{1, \infty}} = \| f \|_{L^{\infty}} + \| \partial_x f \|_{L^{\infty}} \,.
$$
\begin{lemma}\label{lemma:boundTsu}
Let $\beta >1$. We have for $s > 1/2$  
\begin{equation}\label{H2Bis}
\left| \frac{d}{dt} \|P_N\Phi_t^Nu\|^2_{ H^{s+\frac\beta2}} \right|
\lesssim  \|  P_N \Phi^N_t   u\|_{W^{1, \infty}}   \| P_N \Phi^N_t   u\|^2_{H^s}  \,,
\end{equation}
\begin{multline}\label{H1Bis}
%&
 \left| \frac{d}{dt} \left(  \| \Phi_tu\|^2_{ H^{s+\frac\beta2}} - \|P_N\Phi_t^Nu\|^2_{ H^{s+\frac\beta2}}  \right) \right|
\\ 
\lesssim  \left(\| \Phi_t u\|_{W^{1, \infty}} + \|  P_N \Phi^N_t u\|_{W^{1, \infty}} \right) 
\left( \| \Phi_t u\|_{H^s} + \|  P_N \Phi^N_t u\|_{H^s} \right)   \| \Phi_t u -  P_N \Phi^N_t u \|_{H^s}
 \\ 
+ \left( \|\Phi_t u\|^2_{H^s} + \| P_N \Phi^N_t u\|^2_{H^s} \right) 
  \|  \Phi_t u - P_N \Phi^N_t u \|_{W^{1, \infty}},
\end{multline}
and for $s > \frac12 + \frac\beta2$
\begin{equation}\label{H2}
\left| \frac{d}{dt} \| P_N \Phi_t^Nu\|^{2r}_{H^s} \right|
\lesssim_r  \| P_N \Phi^N_t   u\|^{2r-2}_{H^s} \|  P_N \Phi^N_t   u\|_{W^{1, \infty}}   \| P_N \Phi^N_t   u\|^2_{H^{s-\frac\beta2}}  \,,
\end{equation}
\begin{multline}\label{H1}
%&   
\left| \frac{d}{dt} \left( \|  \Phi_tu\|^{2r}_{H^s}  -    \| P_N \Phi_t^Nu\|^{2r}_{H^s} \right)   \right|
 \\ 
\lesssim_r \left(\| \Phi_t u\|_{W^{1, \infty}} + \|  P_N \Phi^N_t u\|_{W^{1, \infty}} \right)
\left( \| \Phi_t u\|^2_{H^{s - \frac\beta2}} + \| P_N \Phi^N_t u\|^2_{H^{s - \frac\beta2}} \right)   \Big( 
\big| \|\Phi_t u \|^{2r-2}_{H^s} - \| P_N \Phi^N_t u \|^{2r-2}_{H^s} \big| \Big)  
\\ 
+
\left(\| \Phi_t u\|_{W^{1, \infty}} + \|  P_N \Phi^N_t u\|_{W^{1, \infty}} \right) 
\left( \| \Phi_t u\|_{H^{s - \frac\beta2}} + \| P_N \Phi^N_t u\|_{H^{s - \frac\beta2}} \right)
\\
\times
\left( \| \Phi_t u\|^{2r-2}_{H^{s}} + \| P_N \Phi^N_t u\|^{2r-2}_{H^{s}} \right)
 \|\Phi_t u - P_N \Phi^N_t u\|_{H^{s-\frac\beta2}}
  \\ 
+
\left( \|\Phi_t u \|^{2r-2}_{H^s} + \| P_N \Phi^N_t u \|^{2r-2}_{H^s} \right) \left( \|\Phi_t u \|^2_{H^{s-\frac\beta2}} + \| P_N \Phi^N_t u \|^2_{H^{s-\frac\beta2}}  \right) 
 \|  \Phi_t u - P_N \Phi^N_t u \|_{W^{1, \infty}};
\end{multline}
the estimates hold
for all $t \in \R$
and
for all $N \in \N \cup \{ \infty \}$.
\end{lemma}

\begin{proof}

Note that \eqref{H2} follows by \eqref{H2Bis}. The proof of \eqref{H2Bis} is the same as 
\eqref{GammaSmoothing}, except that we do not estimate the $L^{\infty}$ norm with the $H^{s}$ norm in \eqref{IneffLInf}. 

The arguments leading to \eqref{H1Bis}-\eqref{H1} are essentially the same. We prove \eqref{H1Bis}  
and explain the modifications needed to prove \eqref{H1}.
As in the proof of Proposition \ref{prop:energy}, we use~\eqref{PivotB} with $\sigma =  s$ and $\sigma =  s + \beta/2$ to compute 
the~$\frac{d}{dt}$. We have
\begin{equation}
\pi \frac{d}{dt} \|P_N\Phi_t^Nu\|^2_{ H^{s+\frac\beta2}} 
  :=  -  \int ( M(D_x) P_N \Phi^N_t u )  \partial_x M(D_x) ( (P_N \Phi^N_t u )^2 ),
   \label{eq:GammaNNew1} \\
\end{equation}
where
$$
 M(D_x) := \left( \frac{1 + |D_x|^{2s +\beta}}{1 + |D_x|^{\beta} } \right)^{1/2}\,. \qquad 
$$
We rewrite the integral in \eqref{eq:GammaNNew1} 
\begin{align}\label{OTRHS}
&   2 \int \big( M(D_x) P_N \Phi^N_t u \big)  M(D_x) (   (  \partial_x P_N \Phi^N_t u) P_N \Phi^N_t u)  
\\ \nonumber
&
=   2 \int \big( M(D_x) P_N \Phi^N_t u \big) \big( M(D_x) \partial_x P_N \Phi^N_t u \big) P_N \Phi^N_t u + 
 \big( M(D_x) P_N  \Phi^N_t u \big) \big[ M(D_x), P_N \Phi^N_t u \big] \partial_x P_N \Phi^N_t u\,.
\end{align}
Integrating by parts we can rewrite the first term on the right hand side of \eqref{OTRHS} as
$$
 \int \big( M(D_x) P_N \Phi^N_t u \big) \big( M(D_x)  P_N \Phi^N_t u \big) \partial_x P_N \Phi^N_t u\,.
$$
Thus we have rewritten the right hand side of \eqref{eq:GammaNNew1} as
\begin{equation}\label{UsingThis}
 - \int   \big( M(D_x) P_N \Phi^N_t u \big) \big( M(D_x)  P_N \Phi^N_t u \big) \partial_x P_N \Phi^N_t u 
-   2 \big( M(D_x) P_N \Phi^N_t u \big) \big[ M(D_x), P_N \Phi^N_t u \big] \partial_x P_N \Phi^N_t u  \,.
\end{equation} 
Decomposing a trilinear operator $L(a,b,c)$ as
\begin{equation}\label{trilDecomp}
L(a, b, c) - L(a_N,  b_N,  c_N) =   L(a- a_N, b, c) + L(  a_N, b-b_N, c ) + L( a_N,  b_N, c - c_N ) 
\end{equation}
and using \eqref{eq:GammaNNew1}-\eqref{UsingThis}
we have rewritten the left hand side of \eqref{H1Bis} as 
$$ 
2\pi |A(t) + B(t) + C(t)| ,
$$
where 
\begin{align}\nonumber
A(t)  & =  
\frac{1}{\pi}  \int  \big( M(D_x) (\Phi_t u -  P_N \Phi^N_t u) \big) \big( M(D_x)  \Phi_t u \big) \partial_x   \Phi_t u
 \\
 \nonumber
  & 
  + \frac{2}{\pi} \int   \big( M(D_x) (\Phi_t u -  P_N \Phi^N_t u) \big) \big[ M(D_x),  \Phi_t u \big] \partial_x  \Phi_t u 
\end{align}
and $B(t), C(t)$ are defined in the analogous way, according to the decomposition \eqref{trilDecomp}.
Noting 
\begin{equation}\label{MultEquiv}
M(n) \simeq 1+ |n|^s, 
\end{equation}
we use the Cauchy--Schwartz inequality and the commutator estimates 
estimates \cite{KenigPilod}      
\begin{equation}
\big\| \big[  M(D_x), f \big] g  \big\|_{L^2} \lesssim \| f \|_{W^{1, \infty}} \| g \|_{H^{s-1}} + \| f \|_{H^s} \| g \|_{L^{\infty}} ,
\end{equation}
with $f= \Phi_t u$ and $g = \partial_x  \Phi_t u$ to show
$$
|A(t)| \lesssim \|\Phi_t u -  P_N \Phi^N_t u\|_{H^s}       
 \|\Phi_t u\|_{H^s}  \|\partial_x \Phi_t u \|_{W^{1, \infty}} 
$$ 
which is one of the contributions of the right hand side of \eqref{H1Bis}. The analysis of $B(t), C(t)$ is 
analogous, leading to the other contributions of the right hand side of \eqref{H1Bis}.

To prove \eqref{H1} we proceed similarly, starting from
\begin{equation}
\pi \frac{d}{dt} \| P_N \Phi_t^Nu\|^{2r}_{H^s}  = -
r\|P_N \Phi^N_t u\|^{2r-2}_{H^s}    \int ( M(D_x) P_N \Phi^N_t u )  \partial_x M(D_x) ( (P_N \Phi^N_t u)^2 )\,,\label{eq:GammaNNew2Bis}
\end{equation}
where now
$$
 M(D_x) := \left(  \frac{1 + |D_x|^{2s}}{1+|D_x|^\beta}  \right)^{1/2}.
$$
The only difference is that the 
decomposition \eqref{trilDecomp} has to be slightly modified in order to take into account the factors 
$\|\Phi_t u\|^{2r-2}_{H^s}, \|P_N\Phi^N_t u\|^{2r-2}_{H^s}$
(like in \eqref{DecFinalmente2}-\eqref{DecFinalmente3}). This is however straightforward.
The important thing to note is that the $H^{s - \frac{\beta}{2}}$ norm on the r.h.s. supplants the $H^s$ 
norm (see \eqref{H1}), because \eqref{MultEquiv} is replaced by
$$ M(n) \simeq 1 +  |n|^{s - \frac{\beta}{2}} \, . $$ 
\end{proof}

%%%%%%%%%%%%%%%%%%%%%%%%%%%%%%%%%%%%%%%%%%%%%%%%%%%%%%%%%%%%%%%%%%%%%%%%%%%%%%%%%%%%%%%%%%%%%%%%%%%%%%%%%%%%%%%%%%%%%%%%%%

\section{(Modified) energy estimates for the NLS equation}\label{Sec:DeterministicNLS}
We will work with the truncated equation
\begin{equation}\label{NLSQuintictruncated}
i \partial_t u + \partial_{x}^2 u = P_N (|P_N u|^4 P_Nu), \quad u(0,x)=u_0(x).
\end{equation}
A direct computation shows that mass and energy are still preserved by the truncated flow, that is for all $N \in \N$ and $t \in \R$
\begin{equation}\label{EqNLS:ConsEnN}
\mc E_1(P_N u(t)) = \mc E_1(P_N u(0)), \qquad  \| P_N u(t) \|_{L^2} = \| P_N u(0) \|_{L^2}\,.
\end{equation}
It follows
\begin{equation}\label{EqNLS:ControlH1WithEnergy}
\| P_N u(t) \|_{H^{1}}^2 \lesssim \|P_N u_0\|^2_{L^2} + \mc E_{1}(P_N u_0).
\end{equation}
$\Phi_t^N u $ denotes the flow of \eqref{NLSQuintictruncated} and $\Phi_t u := \Phi_t^{\infty} u $ the one of 
\eqref{NLSQuintic}.
We have the following statement.
\begin{proposition}\label{prop:NLS-dtH}
Let $N \in \N \cup \{ \infty\}$, $k \geq 2$ be an integer and $R>0$. If 
\begin{equation}\label{eq:APrioriEnBound}
\|u_0\|_{L^2} + \mc E_{1}(u_0)\leq R,
\end{equation} 
 any solution of \eqref{NLSQuintictruncated} with initial datum $u_{0}$ satisfies 
\begin{equation}\label{DtH2Growth}
\left| \frac{d}{dt} \| P_N u(t) \|_{H^k}^2 \right| \lesssim_{R, k} 1 + \| P_N u(t)\|_{H^k}^{2}\,.
\end{equation}
\end{proposition}

\begin{proof}
Since $\| P_N u(t) \|_{L^2} \leq  \| u(0) \|_{L^2} \leq R$ (see the second identity in \eqref{EqNLS:ConsEnN}), it suffices to show 
\begin{equation}\label{DtH2GrowthHomogeneous}
\left| \frac{d}{dt} \|  \partial_x^k P_N  u(t) \|_{L^2}^2 \right| \lesssim_{R, k} 1 + \|   P_N  u(t)\|_{H^k}^{2}\,.
\end{equation}
Taking $\partial_x^k$ of \eqref{NLSQuintictruncated} we have 
$$
\partial_t \partial_x^k P_N  u(t) = i  \partial_{x}^k P_N u(t) - i \partial_x^k P_N (|P_N u(t)|^4 P_N u(t))
$$
Multiplying this equation against $\partial_x^k P_N \bar u(t)$, integrating over $dx$ and taking the real part, we get 
$$
\frac{1}{2} \partial_t \int |\partial_{x}^k P_N u(t) |^2 =  \Im \int  \partial_x^k ( |P_N u(t)|^4 P_N u(t)) \partial_x^k P_N \bar u(t) ; 
$$
note that the projector $P_N$ in front of the nonlinearity on the r.h.s. has been removed by orthogonality.
We rewrite the r.h.s. as 
\begin{equation}\label{MainPlusLO}
\Im \int ( \partial_x^k  P_N \bar u(t))^2 (P_N u(t))^3 P_N \bar u(t) + \mbox{lower order terms},
\end{equation}
where, denoting with $v_j$ either $u$ or $\bar u$, the lower order terms are 
a linear combination of monomials of the form 
\begin{equation}\label{TermsII}
\Im \int (\partial_x^k  P_N \bar u(t) ) (\partial_x^{\alpha_1} P_N  v_1(t) ) \dots (\partial_x^{\alpha_5} P_N  v_5(t) )
\end{equation}
where $\alpha_{m} \in \N \cup \{ 0 \}$ and for $M \in \{ 2, \ldots, 5\}$ we have
$$
\alpha_1 + \dots + \alpha_M = k, \quad 1 \leq  \alpha_1, \ldots \alpha_M \leq k - 1, \quad \alpha_{M+1} \ldots, \alpha_5 = 0
$$ 
(the third condition is empty for $M=5$).

The first term in \eqref{MainPlusLO} is bounded using the H\"older inequality 
and the embedding 
$H^{1} \hookrightarrow L^{\infty}$, recalling that the $H^{1}$ norm of $P_N u(t)$ is controlled by \eqref{EqNLS:ControlH1WithEnergy}-\eqref{eq:APrioriEnBound}
$$
\left| \int ( \partial_x^k  P_N \bar u(t))^2 (P_N u(t))^3 P_N \bar u(t) \right|
\leq \| P_N  u(t) \|_{H^{k}}^2 \| P_N  u(t) \|_{L^{\infty}}^4 \lesssim_R \| P_N  u(t) \|_{H^{k}}^2\,.
$$

To estimate the lower order terms, we use the Gagliardo-Nirenberg inequality (recall that $\alpha_m \geq 1$ for $m =1, \ldots, M$)
$$
\|  \partial_{x}^{\alpha_m} P_N u(t) \|_{L^{2M}} \leq 
\|  P_N u(t) \|_{H^{k}}^{\theta_{m}} \|   P_N u(t) \|_{H^{1}}^{1- \theta_{m}}  , \quad
\theta_m = \frac{\alpha_m - \frac{1}{2} - \frac{1}{2M}}{k-1},
\quad m =1, \ldots, M.
$$
Again, using H\"older's inequality, controlling $\partial_x^k  P_N \bar u(t)$ in $L^2$, the terms
$\partial_{x}^{\alpha_m} P_N u(t)$, $m =1, \ldots, M$ in $L^{2M}$ and
the terms $\partial_{x}^{\alpha_m} P_N u(t)$, $m =M+1, \ldots, 5$ in $L^{\infty}$, we can proceed as before to get the desired bound 
\eqref{DtH2GrowthHomogeneous} for each lower order term, as long as 
$$
\sum_{m=1}^{M} \theta_m \leq 1.
$$  
Recalling that $\sum_{m=1}^{M} \alpha_m \leq k$ and $M \geq 2$, is immediate to check that this holds (the inequality is in fact strict), 
so the proof is concluded. 
\end{proof}

We also need some related bounds, contained in the following Lemma.

\begin{lemma}\label{lemma:boundTsuNLS}
Let $N \in \N \cup \{ \infty\}$, $k, r \geq 1$ ben integers and $R>0$. If 
\eqref{eq:APrioriEnBound} holds, then there is $\ell \in\N$ and $0 < \mu_1 \leq \mu_2 < 2$ such that
\begin{align}\label{H1NLS}
&  \left| \frac{d}{dt} \left(  \|   \Phi_t u \|_{H^{k}}^{2r} - \| P_N   \Phi_t^N u \|_{H^{k}}^{2r} \right)  \right|
\\ \nn
&\lesssim_{r, R, k}  (\|   \Phi_t u \|_{H^{k}}^{r - 1} -    \| P_N  \Phi_t^N u \|_{H^{k}}^{r - 1} )  ( 1 + \|   \Phi_t u \|_{H^{k}}^{\ell} + \| P_N  \Phi_t^N u \|_{H^{k}}^{\ell})
\\ \nonumber
&+  \sum_{j = 1, 2} \|  \Phi_t u - P_N  \Phi_t^N u\|_{H^{k}}^{\mu_j}  ( 1 + \|   \Phi_t u \|_{H^{k}}^{\ell} + \| P_N  \Phi_t^N u \|_{H^{k}}^{\ell}).
\end{align}
Moreover we have 
\begin{equation}\label{H2NLS}
\left| \frac{d}{dt}   \| P_N  \Phi_t^N u \|_{H^{k}}^{2r}   \right|
\lesssim_{r,R,k}  1 + \| P_N  \Phi_t^N u \|_{H^{k}}^{2r} \,,
\end{equation}
for all $N \in \N \cup \{ \infty \}$.
\end{lemma}

\begin{proof}
Since 
$$
\frac{d}{dt}   \| P_N \Phi_t^N u \|_{H^{k}}^{2r} = r   \| P_N \Phi_t^N u \|_{H^{k}}^{2r - 2} \frac{d}{dt}   \| P_N \Phi_t^N u \|_{H^{k}}^{2} 
$$
the estimate \eqref{H2NLS} follows by Proposition \ref{prop:NLS-dtH}. Thus we need to prove \eqref{H1NLS}. 
We write 
$$
\frac{d}{dt}   \| P_N \Phi_t^N u \|_{H^{k}}^{2r} = r   \| P_N \Phi_t^N u \|_{H^{k}}^{2r - 2} \Big( \frac{d}{dt}   \| P_N \Phi_t^N u \|_{L^{2}}^{2} +
\frac{d}{dt}   \| \partial_x^k P_N \Phi_t^N u \|_{L^2}^{2} \Big). 
$$
We will focus on the contribution to \eqref{H1NLS} coming from second addendum, namely   
\begin{equation}\label{NontrivialDiff}
  \|  \Phi_t u \|_{H^{k}}^{2r - 2} 
\frac{d}{dt}   \| \partial_x^k  \Phi_t u \|_{L^2}^{2}  
-
   \| P_N \Phi_t^N u \|_{H^{k}}^{2r - 2}  \frac{d}{dt}   
   \| \partial_x^k P_N \Phi_t^N u \|_{L^2}^{2} ;
\end{equation}
the analysis of the contributions involving $\frac{d}{dt} \| P_N \Phi_t^N u \|_{L^2}^{2}$
is in fact easier. 
Proceeding as in the proof of  Proposition \ref{prop:NLS-dtH} 
we see that $\frac{d}{dt}   \| \partial_x^k  P_N \Phi_t^N u \|_{L^{2}}^{2}$ is a linear combination
 of  
monomials of the form (we may assume there are $m$ such monomials) 
\begin{equation}\label{TermsIIBis}
L \left(v_1^N(t), \ldots, v_6^N(t) \right) :=  \Im \int (\partial_x^{k} v_1^N(t) ) \dots (\partial_x^{\alpha_6} v_6^N(t) )\,,
\end{equation}
where $\alpha_{j} \in \N \cup \{ 0 \}$ are such that
$$
\alpha_2 + \dots + \alpha_6 = k, \quad, \alpha_2 \leq k, \quad \alpha_3, \ldots, \alpha_6 \leq  k-1
$$ 
and $v_j^N(t)$ can be either $P_N \Phi_t^N u$ or its conjugate. Thus we can bound 
\eqref{NontrivialDiff}
with the modulus of a linear combination ot
 of terms of the form
(here we set $v(t)=v^{\infty}(t)$)
$$
\| v (t) \|_{H^{k}}^{2r - 2} L \left(v_1 (t), \ldots, v_6(t) \right) -
\| v^N (t) \|_{H^{k}}^{2r - 2} L \left(v_1^N(t), \ldots, v_6^N(t) \right)\,.
$$
Decomposing  
\begin{align}\label{DecFinalmente}
& L\left(v_1 (t), \ldots, v_6(t) \right) - L\left(v_1^N (t), \ldots, v^N_6(t) \right) 
\\ \nonumber
&=L\left(v_1 (t) - v_1^N(t), \ldots, v_6(t) \right) +
L\left(v_1^N (t), v_2(t) - v_2^N(t), \ldots, v_6(t) \right) + \dots
\\ \nonumber
& \qquad \qquad  \dots + L\left(v_1^N (t),  \ldots, v_5^N (t), v_6(t) -v_6^N(t) \right),
\end{align}
 we have reduced to bound the modulus of terms of the form
\begin{equation}\label{DecFinalmente2}
A(t) + B(t) + C(t) + D(t) + E(t) + F(t) + G(t),   
\end{equation}
where 
\begin{equation}\label{DecFinalmente3}
A(t) =  (  \| v (t) \|_{H^{k}}^{r - 1} -  \| v^N (t) \|_{H^{k}}^{r - 1} ) (  \| v (t) \|_{H^{k}}^{r - 1} +  \| v^N (t) \|_{H^{k}}^{r - 1} ) L\left(v_1 (t), \ldots, v_6(t) \right) ,
\end{equation}
\begin{equation}\nonumber
B(t)   =  \| v^N (t) \|_{H^{k}}^{2r - 2} L\left(v_1 (t) - v_1^N (t) , v_2(t), \ldots, v_6(t) \right)
\end{equation}
and $C(t), \ldots, G(t)$ are defined in the analogous way to $B(t)$, according to the decomposition of \eqref{DecFinalmente}. 
Starting by this decomposition, is straightforward to note that, estimating the terms as in the proof of Proposition \ref{prop:NLS-dtH}, we arrive to the bound 
\eqref{H1NLS}.
\end{proof}

From \cite[Proposition 2.2]{PTV}, proceeding for instance as in \cite[Proposition 2.7]{sigma}, we obtain also the following property. 

\begin{lemma}\label{lemma:NLS-diff-flows}
Let $\s>1$ $R>0$, and $K\subset B_{\s}(R)$ be a compact set. Let further $t\in\R$. For all $\e>0$ there is $\bar N\in\N$ such that for all $N\geq \bar N$
\be
\sup_{u\in K}\|\Phi_tu-\Phi^N_tu\|_{H^{\s}}<\e\,. 
\ee 
\end{lemma}

A crucial result for our analysis is \cite[Theorem 1.4]{PTV}. We report it below.

\begin{theorem}[Planchon, Tzvetkov, Visciglia \cite{PTV}]\label{th:PTV}
Let $k \geq 2$ an integer. There are $m_0\in\N$, $C>0$ and $R_{2k}(u)\,:\,H^{2k-1}(\T)\to\R$ with $R_{2k}(0)=0$ and
\be\label{eq:LipRk}
|R_{2k}(u)-R_{2k}(v)|\leq C\|u-v\|_{H^{2k-1}}(1+\|u\|^{m_0}_{H^{2k-1}}+\|v\|^{m_0}_{H^{2k-1}})\,,
\ee
such that for all $N\in\N\cup\{\infty\}$
\bea\label{1Smoothing}
\left|\frac{d}{dt}\left(\frac12\|P_N\Phi^N_tu\|^2_{H^{2k}}+R_{2k}(P_N\Phi^N_tu)\right)\right|
%&
\leq
%&
C(1+\|P_N\Phi^N_tu\|^{m_0}_{H^{2k-1}})\,,\quad k\geq2\,.
\eea
\end{theorem}
Set
$$
\mc E_{2k}(u):=\frac12\|u\|^2_{H^{2k}}+R_{2k}(u)\,.
$$
The following statement is a corollary of Theorem~\ref{th:PTV}.
\begin{lemma}\label{lemma:tentative}
Let $k\geq 2$ be an integer. There are $m_0\in\N$, $C>0$ such that 
\begin{equation}\label{eq:tentative}
\left|\frac{d}{dt}  \left( \mathcal{E}_{2k} ( \Phi_t u) - \mathcal{E}_{2k} ( P_N \Phi_t^N u) \right) \right| 
\leq C \|  \Phi_t u - P_N \Phi_t^N u \|_{H^{2k-1}}
 (1+ \|  \Phi_t u \|^{m_0}_{H^{2k-1}} + \| P_N \Phi_t^N u \|^{m_0}_{H^{2k-1}})
\end{equation}
for all $t \in \R$ and $N \in \N$.
\end{lemma}
\begin{proof}
Using the identity (107) in \cite{PTV} 
we have that for all $N \in \N \cup\{ \infty\}$ there are $\ell_k,\ell_k^*\in\N$ and multilinear  
forms $\{L_{\ell, 2k}\}_{\ell=1,\cdots,\ell_k}$, $\{L^*_{\ell, 2k}\}_{\ell=1,\cdots,\ell_k^*}$ such that
\begin{align}
\frac{d}{dt} \mathcal{E}_{2k} ( P_N \Phi_t^N u)  = \sum_{\ell=1}^{\ell_k} L_{\ell, 2k} 
(v^{N}_1(t), \ldots v^{N}_{\ell}(t)) + 
\sum_{\ell = 1}^{\ell_k^*} L^*_{\ell, 2k} 
(v^{*,N}_1(t), \ldots v^{*,N}_{\ell}(t)) , 
\end{align}
where each $v^{N}_j(t)$ ($j=1,\ldots,\ell$) can be either $P_N \Phi_t^N u$ or its conjugate, while 
each $v^{*,N}_j(t)$ can be be either $(1- P_N) \Phi_t^N u$ or its conjugate. Moreover the analysis of 
\cite{PTV} implies that
\bea
|L_{\ell, 2k} 
(v^{N}_1(t), \ldots v^{N}_{\ell}(t))| &\lesssim&  \prod_{j=1}^{\ell} \| v^{N}_j(t) \|_{H^{2k-1}}
\label{PTVProp1}\,,\\
|L^*_{\ell, 2k} 
(v^{*,N}_1(t), \ldots v^{*,N}_{\ell}(t))| &\lesssim &  \prod_{j=1}^{\ell} \| v^{*,N}_j(t) \|_{H^{2k-1}}\,.\label{PTVProp2}
\eea
Thus we have 
\begin{align}\label{LongPTV}
 \frac{d}{dt}\Big(\mathcal{E}_{2k} ( \Phi_t u) & - \mathcal{E}_{2k} ( P_N \Phi_t^N u) \Big)=
 \\ \nonumber
& \sum_{\ell=1}^{\ell_k}  
\left( L_{\ell, 2k}  (v_1(t), \ldots v_{\ell}(t)) - L_{\ell, 2k}(v^{N}_1(t), \ldots v^{N}_{\ell}(t)) \right) + 
\sum_{\ell = 1}^{\ell_k^*} L^*_{\ell, 2k} 
(v^{*,N}_1(t), \ldots v^{*,N}_{\ell}(t)), 
\end{align}
where we are denoting $v_j:= v_j^{\infty}$.
The contribution of the $L_{\ell, 2k}$ terms can be controlled using the bound \eqref{PTVProp1} 
and a multilinear decomposition like the 
one in \eqref{DecFinalmente}.   
The contribution of the $L^*_{\ell, 2k}$ terms can be controlled using the bound \eqref{PTVProp2}. 
In both cases we have estimates compatible with the r.h.s. of \eqref{eq:tentative}.
\end{proof}

%%%%%%%%%%%%%%%%%%%%%%%%%%%%%%%%%%%%%%%%%%%%%%%%%%%%%%%%%%%%%%%%%%%%%%%%%%%%%%%%%%%%%%%%%%%%%%%%%%%%%%%%%%%%%%%%%%%%%%%%%%%%%%%%%%%%%%%%%%%%%%%%%%%%%%%%%%%%%%%%%%%%%%%%%%%%%%%%%%%%%%%

\section{Proof of Theorem \ref{TH:Invariance}}\label{sect:invarianza}
The goal of this section is to prove Theorem~\ref{TH:Invariance}. The main difficulty is that the flow of the BBM equation $\Phi_t$ is not well defined on $H^s$, $s < \frac{\beta}{2}-\frac12$, which  is the regularity of $\g_{0}$-typical functions. However, as proved in Section \ref{sect:LWP}, we can define locally in time a flow on $C^{\alpha}$
for $0 < \alpha < \frac{\beta}{2} - \frac{1}{2}$, which is a full measure set for $\g_0$. The local existence time depends on the $C^{\alpha}$
norm of the initial datum. The goal of this section is to prove that  the local flow can be promoted to a global one~$\g_{0}$-almost surely using 
the invariance of the $\g_{0}$ measure under the approximated flow~$\Phi_t^N$. Once this is achieved, it is easy to prove that   
$\g_{0}$ is invariant under $\Phi_t$ (now defined~$\g_{0}$-almost surely). 

The 
invariance of $\g_{0}$ under the approximated flow~$\Phi_t^N$ is a consequence of the conservation of the $H^{\beta/2}$ norm, as defined in
\eqref{Def:SobNorm}. Indeed, 
using \eqref{cons} we can readily prove the following

\begin{proposition}\label{prop:quasi-invNBis}
It is for any measurable set $A$ and for all $N \in \N \cup \{ \infty\}$
\be\label{eq:invN}
\g_{0} (\Phi_t^N (A)) = 
\g_{0} ( A)  .
\ee
\end{proposition}  

\begin{proof}
Let $\bar t \in \R$. 
Proceeding as in the proof of forthcoming Proposition \ref{prop:quasi-invN}, replacing the measure $\r_{s,N}$ with $\g_{0}$, we arrive to 
\begin{equation}
\frac{d}{d t}  \left(  \g_{0}  \circ \Phi_t^N (A) \right)\Big|_{t=\bar t}
= \frac12 \int_{\Phi^N_{\bar t}(A)} \g_{0} (du)  \frac{d}{dt}
\| P_N \Phi_t^N u\|^2_{H^{\frac\beta2}}\Big|_{t=0}  = 0\,,
\end{equation}
where we used $\frac{d}{dt} \| P_N \Phi_t^N u\|^2_{H^{\frac\beta2}} = 0$; see \eqref{cons}.
\end{proof}

We are now ready to prove Theorem \ref{TH:Invariance}.

\begin{proof}[Proof of Theorem \ref{TH:Invariance}]
Let $T>0$, $K>0$. We partition $[0,T]$ into~$J$ intervals of size at most $$\tau_K :=\frac{c}{K+1},$$ where $c$ is small enough that 
we have local well-posedness in $C^{\alpha}$ for all times $t \in [0, \frac{c}{K+1}]$ and all data in $B^{\alpha}(K)$ (see \eqref{eq:cialfaball}). %; see Section~\ref{sect:LWP}. 
Clearly 
 \begin{equation}\label{J}
 J \leq c^{-1} T (K +1) +1 . 
 \end{equation}
We set 
\begin{align}\nonumber
 E_{K,N,T}
& :=  \Big\{ u \notin  B^{\alpha}(K/2) \Big\} \cup  \Big\{ u \notin  \Phi^N_{- \tau_K}(B^{\alpha}(K/2)) \Big\}  \cup \Big\{  u \notin  \Phi^N_{- 2 \tau_K}(B^{\alpha}(K/2)) \Big\}
\\ \label{Def:Except} 
&
\quad \quad \quad 
\ldots \cup \Big\{  u \notin  \Phi^N_{- (J-1) \tau_K }(B^{\alpha}(K/2)) \Big\}
\cup 
\Big\{  u \notin  \Phi^N_{-J\tau_K}(B^{\alpha}(K/2)) \Big\} \,.
\end{align}
We will show that the $\g_{0}$ measure of these sets vanishes in the limit $K\to \infty$ (and $\t_K\to0$). 
To do so we have to take advantage of the invariance of $\g_{0}$. Indeed by the classical estimate
$$
  \g_{0} (B^{\alpha}(K/2)^C)  \leq Ce^{-cK^2}, 
$$
we have
\begin{align}\label{TTT}
\g_{0} (E_{K,N,T}) 
&
\leq \sum_{j=0}^{J} \g_{0}( \Phi^N_{- \tau_K}(B^{\alpha}(K/2))^C ) =
\sum_{j=0}^{J} \g_{0} (B^{\alpha}(K/2))^C ) \lesssim J e^{-cK^2} \lesssim TK e^{-cK^2} \, , 
\end{align}
where we used \eqref{eq:invN} and then \eqref{J} in the last inequality

Let $\{N_K\}_{K\in\N}$ be a diverging sequence and 
\be\label{eq:inclusione-fondamentale}
E_T := \bigcap_{K \in \N}  E_{K, N_K,T} \,.
\ee
Using \eqref{TTT} and Proposition \ref{ApproxThm}, we will first show that given $T>0$, the flow $ \Phi_{t}$ is well defined for $t \in [0, T]$ and for all initial data in $E_T^C$
and that $ \g_{0}( E_T) = 0$. 
Once we have that, the statement follows simply 
removing the set $\bigcup_{T \in \Z}  E_{T}$ which has zero~$\g_{0}$-measure (negative times are covered just by time reversibility).  

Let us consider
\begin{align}\label{eq:setsA}
 E_{K,N,T}^C
& :=  \Big\{ u \in  B^{\alpha}(K/2) \Big\} \cap  \Big\{ u \in  \Phi^N_{- \tau_K}(B^{\alpha}(K/2)) \Big\}  \cap \Big\{  u \in  \Phi^N_{- 2 \tau_K}(B^{\alpha}(K/2)) \Big\}
\\ \nonumber 
&
\quad \quad \quad 
\ldots \cap \Big\{  u \in  \Phi^N_{- (J-1) \tau_K }(B^{\alpha}(K/2)) \Big\}
\cap 
\Big\{  u \in  \Phi^N_{-J\tau_K}(B^{\alpha}(K/2)) \Big\} \,.
\end{align}Since 
$$
 \Phi^N_{j \tau_K} E_{K, N,T}^C \subset  B^{\alpha}(K/2),
\qquad j=0, \ldots, J+1,
$$
by the group property of the flow we can apply Proposition \ref{LocWPCalpha} on each
time interval $[j \tau_K, (j+1) \tau_K]$ so that we have  
$$ 
\sup_{t \in [0, T]} \sup_{N \in \mathbb{N}} \sup_{\phi (\cdot, 0) \in  E_{K,N,T}^C } \|  \Phi^N_{t} \phi (\cdot, 0) \|_{C^{\a}} \leq K, 
\qquad j=0, \ldots, J+1 \, .
$$
Thus, for all $K >1$ we can pick $N_K$ sufficiently large 
and invoke Proposition~\ref{ApproxThm},
so that we deduce that also $\Phi_t$ is well defined for times $t \in [0,T]$ and data in 
$$ 
E_T^C = \bigcup_{K \in \N} E_{K, N_K,T}\,.
$$  
On the other hand,
by \eqref{TTT}
we have
$$\g_{0}(E_T) = \g_{0} \left( \bigcap_{K \in \N}  E_{K, N_K,T} \right) \leq \lim_{K \to \infty} \g_{0} (  E_{K, N_K,T} )  = 0,$$ 
as claimed.

This proves the first part of the statement. The second part is to show that $\g_{0}$ is invariant under the flow $\Phi_t$.
We already know that 
$\g_{0}$ is invariant under the approximated flow $\Phi_t^N$, for all~$N \in \N$. The goal is to pass to the limit $N \to \infty$.

Because of time reversibility,
it suffices to show that for all Borel sets $A$ such that
$$
A \subset \bigcup_{T \in \Z} E_T^C , \quad \quad (\mbox{$\Phi_t$ is well defined on $\bigcup_{T \in \Z} E_T^C$})
$$
we have 
\begin{equation}\label{QuasiInvFinal}
\g_{0}(\Phi_t (A)) \leq  \g_{0} (A)  \quad\forall t \in \R\,.
\end{equation}

In fact it suffices to prove \eqref{QuasiInvFinal} for compact sets, then the general case 
follows by the inner regularity of $\g_{0}$. 
Let $A$ be compact in the $C^{\alpha}$ topology and thus, in particular, $A \subseteq B^{\alpha}(K)$ for
 some~$K \in \N$.
Using \eqref{approxProp}
we know that 
\begin{equation}\label{Impl1}
\sup_{t \in [0, T]}
\|  \Phi_{t} u -  \Phi^{N_K}_{t} u \|_{C^{\upsilon}} \leq \varepsilon,
\qquad \upsilon < \alpha .
\end{equation}
This implies
\begin{equation}\nonumber
 \Phi_t ( A ) \subseteq   \Phi^N_t ( A )  + B^\upsilon( \varepsilon ) ,
\quad  |t| \leq T  
\end{equation}
and so
\begin{equation}\label{Impl2}
\gamma_0( \Phi_t ( A ) ) \leq \gamma_0 \left(  \Phi^N_t ( A )  + B^\upsilon( \varepsilon ) \right),
\quad  |t| \leq T  
\end{equation}
Since $\Phi^N_t ( A )$ is 
compact, we have $\Phi^N_t ( A ) = \bigcap_{\varepsilon >0} (\Phi^N_t ( A )  + B^\upsilon( \varepsilon ) )$,
so that we can pass to the limit $\varepsilon \to 0$ in \eqref{Impl2} using the dominated convergence theorem,
getting
$$
\gamma_0( \Phi_t ( A ) ) \leq \gamma_0 \left(  \Phi^N_t ( A )   \right).
 $$ 
Then the invariance of
$\g_{0}$ under the flow of $\Phi_{t}^N$ (see Proposition \ref{prop:quasi-invNBis}) 
implies \eqref{QuasiInvFinal}, as desired.
\end{proof}

%%%%%%%%%%%%%%%%%%%%%%%%%%%%%%%%%%%%%%%%%%%%%%%%%%%%%%%%%%%%%%%%%%%%%%%%%%%%%%%%%%%%%%%%%%%%%%%%%%%%%%%%%%%%%%%%%%%%%%%%%%%%%%%%%%%%%%%%%%%%%%%%%%%%%%%%%%%%%%%%%%%%%%%%%%%%%%%%%%%%%%%
\section{Quasi-invariant measures for the BBM equation}\label{sect:quasi}

In this section we prove the first part of Theorem \ref{TH:quasi}. The proof of the formula \eqref{eq:Densities} for the density is postponed to Section \ref{sec:density}.
Define the family of measures for $N\in\N$
\be\label{eq:rho} 
\r_{s,N}(du):=\exp(-\|P_Nu\|^{2r}_{H^s}) 1_{\{\| u \|_{H^{\beta/2}} \leq R\}}\g_{s}(du)\,,\quad r>2\,. 
\ee
One readily verifies that 
$\r_{s,N}(du)$ converges to $\rho_{s}$, defined in (\ref{Def:Rho}), as $N\to\infty$ (the density of $\r_{s,N}(du)$ w.r.t. $\g_s$ converges to the one of $\rho_{s}$ in $L^1(\g_s)$). Note that for all $N$ the measures $\r_{s,N}(du)$ and $\r_{s}(du)$ depend on the parameters $R,r,\beta$, even though in the sequel we will systematically omit that in the notation.
We denote by
$$
E_N:=\Span_\R\{(\cos(nx),\sin(nx))\,,\quad |n|\leq N \} \,.
$$
Note $\dim E_N = 2N +1$.
 $E_N^{\perp}$ the orthogonal complement of $E_N$ in the topology of $L^2(\T)$. 
Letting $\g_{s, N}^{\perp}$ the measure induced on $E_N^{\perp}$ by the map (recall \eqref{Def:GaussmEasure})
\begin{equation}\label{Def:gammaK2}
\varphi_s(\omega,x)=\sum_{|n| > N}\frac{g_n(\omega)}{(1+|n|^{2s+\beta})^{\frac{1}{2}}}e^{inx},
\end{equation}
the measure $ \g_{s}$ factorises over $E_N \times E_N^{\perp}$ as
\begin{equation}\label{GammaPerp}
 \g_{s} (du) := \frac{1}{Z_N} e^{-\frac12  \| P_N u \|_{H^{s+\beta/2}}^2 } L_N (d P_N u) \,  \g_{s, N}^{\perp}(d P_{>N} u),
\end{equation}
where $L_N$ is the Lebesgue measure induced on $E_N$ by the isomorphism between~$\R^{2N+1}$ and $E_N$ and 
$Z_N$ is a renormalisation factor. This factorisation is useful since we know by \cite[Lemma 4.2]{sigma} that the Lebesgue measure $L_N$ is invariant under
$\Phi_t^N P_N = P_N \Phi_t^N  $. 

\begin{proposition}\label{prop:quasi-invN}
Let $\beta\in(1,2]$, $s>\beta/2$, $r > 2$ and $s + \beta/2 > 3/2$.  Let $\alpha>0$ such that
\begin{equation}
1+\alpha<\beta,\quad s-\alpha>\beta/2\,.
\end{equation}
Let $\s = \s(r,\a,\beta,s) \in (0,1)$ be defined as
$$
\s := 
\max \left(1 - \frac{\a}{r(2s-\beta)}, \,  \frac{1}{2} + \frac{1}{r} \right)
$$
For any measurable set $A$ and for all $t \in \R$ it holds 
\be\label{eq:quasi-invN}
\r_{s,N}(\Phi_t^N (A))\leq \r_{s,N}(A)\exp\left(p\log(1+ c(R, r) |t| p^{\s-1}(\r_{s,N}(A))^{-\frac1p})\right), 
\qquad p \in [1, \infty).
\ee
\end{proposition}
%%%
\begin{proof}
Using the definition \eqref{eq:rho}, the factorisation \eqref{GammaPerp} and Proposition 4.1 of \cite{sigma}, we have for all measurable $E$  
\begin{align}\label{STTG}
& \r_{s,N} \circ \Phi_t^N (E)  = \int_{\Phi_t^N (E)\cap B^{\frac\beta2}(R)} \g_{s}(du) \exp(-\|P_N u\|^{2r}_{H^s}) \\ \nonumber
& = \int_{E\cap B^{\frac\beta2}(R)}  L_N (d P_N u)   \g_{s, N}^{\perp} (d P_{>N} u)
\exp(-\| P_N \Phi_t^N u\|^{2r}_{H^s})  \exp \left( -\frac{1}{2} \| P_N \Phi_t^N u  \|_{H^{s+\frac{\beta}{2}}}^2  \right)  
\\ \nonumber
& =
\int_{E\cap B^{\frac\beta2}(R)} \g_{s}(du) \exp(-\|P_N \Phi_t^N u\|^{2r}_{H^s}) 
\exp\left(\frac12 \| P_N u \|^2_{H^{s+\frac{\beta}{2}}}-\frac12\| P_N \Phi_t^N u\|^2_{H^{s+\frac{\beta}{2}}}\right)
\\ \nonumber
& =
\int_{E}  \r_{s,N}(du) \exp(\| P_N u \|^{2r}_{H^{s}} - \|P_N \Phi_t^N u\|^{2r}_{H^s}) 
\exp\Big(\frac12 \| P_N u \|^2_{H^{s+\frac{\beta}{2}}}-\frac12\| P_N \Phi_t^N u\|^2_{H^{s+\frac{\beta}{2}}}\Big),
\end{align}
%$$
where we used that the Jacobian determinant $|\det D P_N \Phi_t^N (u)|=1$ (see \cite[Lemma 4.2]{sigma}) and in the  
second identity we used \eqref{cons}. 
Using that 
$$t \in (\mathbb{R}, +) \to \Phi_t^N$$ 
is a one parameter group of transformations, we can easily check that
\begin{equation}\label{EqualityAtT=0}
\frac{d}{d t} \left( \r_{s,N} \circ \Phi_t^N (A) \right)\Big|_{t=\bar t} = 
\frac{d}{d t}\left( \r_{s,N} \circ \Phi_t^N (\Phi_{\bar t}^N A ) \right) \Big|_{t=0} \, .
\end{equation}
Using \eqref{STTG}-\eqref{EqualityAtT=0} 
under the choice $E= \Phi_{\bar t}^N A$, we arrive to
\begin{align}
\frac{d}{d t} & \left( \r_{s,N} \circ \Phi_t^N (A) \right)\Big|_{t=\bar t}
\nn \\ 
&=\frac{d}{dt}\int_{\Phi^N_{\bar t}(A)}  \r_{s,N}(du) \exp(\| P_N u \|^{2r}_{H^{s}} - \|P_N \Phi_t^N u\|^{2r}_{H^s}) 
\exp\left(\frac12 \| P_N u \|^2_{H^{s+\frac{\beta}{2}}}-\frac12\| P_N \Phi_t^N u\|^2_{H^{s+\frac{\beta}{2}}}\right) \Big|_{t=0}\nn\\ \label{FirstSummand}
&=-\int_{\Phi^N_{\bar t}(A)} \r_{s,N}(du) \left(r \|P_Nu\|^{2r-2}_{H^s}\frac{d}{dt}\|P_N \Phi_t^Nu\|^{2}_{H^s}\Big|_{t=0}+\frac12\frac{d}{dt}
\| P_N \Phi_t^N u\|^2_{H^{s+\frac\beta2}}\Big|_{t=0}\right)\,.
\end{align}

Now by Proposition~\ref{prop:growth-Hs-norm} (recall $\theta := \frac{2\a}{2s-\beta}$) we have
\bea
&&\left| \int_{\Phi^N_{\bar t}(A)} \r_{s,N}(du) \|P_Nu\|^{2r-2}_{H^s}\frac{d}{dt}\| P_N \Phi_t^Nu\|^{2}_{H^s}\Big|_{t=0} \right| \nn\\
&\lesssim&R^{1+\theta} \int_{\Phi^N_{\bar t}(A)} \r_{s,N}(du) \|P_Nu\|^{2r-\theta}_{H^s}\nn\\
&\leq& (1 + R^2 ) \left(\int_{\Phi^N_{\bar t}(A)} \r_{s,N}(du) \|P_Nu\|^{(2r-\theta)p}_{H^s}\right)^{\frac1p} (\r_{s,N}(\Phi^N_{\bar t}(A)))^{1-\frac1p}\,,\nn
\eea
where we used the H\"older inequality in the last line. 
We have

\begin{align}\label{MainCalculation}
& 
\left(\int_{\Phi^N_{\bar t}(A)} \r_{s,N}(du) \|P_Nu\|^{(2r-\theta)p}_{H^s}\right)^{\frac1p}
\\ \nonumber
& =
\left(\int_{\Phi^N_{\bar t}(A)\cap B^{\frac\beta2}(R)}  \|P_Nu\|^{(2r-\theta)p}_{H^s} \exp(-\|P_Nu\|^{2r}_{H^s}) \g_{s}(du) \right)^{\frac1p}, 
\\ \nonumber
& \leq 
\left(\int\left( \sup_{x \geq 0} x^{(2r-\theta)p} \exp^{-x^{2r}} \right) \g_{s}(du) \right)^{\frac1p},  
\\ \nonumber
& 
\leq  \sup_{x\geq0} x^{2r-\theta}e^{-\frac{x^{2r}}{p}}\lesssim p^{1-\frac{\theta}{2r}}\,,
\end{align}

so that
\begin{align}\label{eq:use1}
\left| \int_{\Phi^N_{\bar t}(A)} \r_{s,N}(du) \|P_Nu\|^{2r-2}_{H^s}\frac{d}{dt}\| P_N \Phi_t^Nu\|^{2}_{H^s}\Big|_{t=0} \right|
&\lesssim R^{1+\theta}p^{1-\frac{\theta}{2r}}(\r_{s,N}(\Phi^N_{\bar t}(A)))^{1-\frac1p}
\\ \nonumber
&\leq R^{1+\theta}p^{\sigma}(\r_{s,N}(\Phi^N_{\bar t}(A)))^{1-\frac1p}\,.
\end{align}
To bound 
$$
\left| \int_{\Phi^N_{\bar t}(A)}\r_{s,N}(du)\frac{d}{dt}\| P_N \Phi_t^N u\|^2_{H^{s+\frac\beta2}}\Big|_{t=0} \right|
$$
we use Proposition~\ref{prop:energy}, so that
\bea
&&
\left| \int_{\Phi^N_{\bar t}(A)}\r_{s,N}(du)\frac{d}{dt}\| P_N \Phi_t^N u\|^2_{H^{s+\frac\beta2}}\Big|_{t=0} \right| \nn\\
& \leq &\int_{\Phi^N_{\bar t}(A)} \r_{s,N}(du)\| P_N u\|^3_{H^{s}}\label{eq:first}\\
&+&\int_{\Phi^N_{\bar t}(A)} \r_{s,N}(du)\| P_N u\|^2_{H^{s}}\| P_N \partial_xu\|_{L^{\infty}}\label{eq:second}\,.
\eea
We use the H\"older inequality for the first and the second summand. For the first summand we have (proceeding as in \eqref{MainCalculation})
\bea
(\ref{eq:first})&\leq&\left(\int_{\Phi^N_{\bar t}(A)} \r_{s,N}(du) \|P_Nu\|^{3p}_{H^s}\right)^{\frac1p} (\r_{s,N}(\Phi^N_{\bar t}(A)))^{1-\frac1p}\nn\\
&\lesssim & \left( \sup_{x\geq0} x^{3}e^{-\frac{x^{2r}}{p}}  \right) \r_{s,N}(\Phi^N_{\bar t}(A)))^{1-\frac1p}
\nn \\
&\lesssim&p^{\frac3{2r}}(\r_{s,N}(\Phi^N_{\bar t}(A)))^{1-\frac1p} 
\nn \\
&
\leq & p^{\sigma}(\r_{s,N}(\Phi^N_{\bar t}(A)))^{1-\frac1p}\,.\label{eq:use2}
\eea
In the last line we used $\frac3{2r} \leq \frac{1}{2} + \frac{1}{r} \leq \sigma$, which is true for $r > 2$.
To handle the second summand, we need to use the estimate
\begin{equation}\label{LinfB}
\Big( \int \| P_N \partial_x u\|_{L^{\infty}}^p \g_{s}(du)  \Big)^{1/p}
\lesssim \sqrt{p} 
\end{equation}
which is valid for $s + \beta/2 > 3/2$. For the proof of \eqref{LinfB} we refer to
\cite[Lemma 6.1]{sigma}. Note that in Lemma 6.1 of \cite{sigma} the estimate is proved for the operator 
$|D_x|^{s+\frac{\beta-1}2-\e}$, $\varepsilon >0$, however is easy to check that the proof works once we replace 
$|D_x|^{s+\frac{\beta-1}2-\e}$ with $\partial_x$, as long as $s + \beta/2 > 3/2$.
Thus, by H\"older's inequality and \eqref{LinfB} we get
\begin{align}\nn
(\ref{eq:second}) & \leq \left(\int_{\Phi^N_{\bar t}(A)} \r_{s,N}(du) \|P_Nu\|^{2p}_{H^s}\|  \partial_x P_N u\|^p_{L^{\infty}}\right)^{\frac1p} 
(\r_{s,N}(\Phi^N_{\bar t}(A)))^{1-\frac1p}
\\ \nn
&\lesssim \left( \sup_{x\geq0} x^{2} e^{-\frac{x^{2r}}{p}} \right)  p^{\frac12}(\r_{s,N}(\Phi^N_{\bar t}(A)))^{1-\frac1p} 
\\ 
&\lesssim   p^{\frac{1}{r} + \frac12}(\r_{s,N}(\Phi^N_{\bar t}(A)))^{1-\frac1p}
\lesssim   p^{\sigma}(\r_{s,N}(\Phi^N_{\bar t}(A)))^{1-\frac1p}   \label{eq:use3}
\,. 
\end{align}

Therefore by (\ref{eq:use1}), \eqref{eq:use2}, (\ref{eq:use3}) we conclude that there is $c(R)$ such that
\be\nn 
\frac{d}{d t} \left( \r_{s,N} \circ \Phi_t^N (A) \right)\leq c(R,r) p^{\s}(\r_{s,N}(\Phi^N_{t}(A)))^{1-\frac1p}\,, 
\ee
From which we get
\be\label{eq:yud2}
\frac{d}{d t} \left( \r_{s,N} \circ \Phi_t^N (A) \right)^{\frac1p}\leq c(R,r) p^{\s-1}\,,
\ee
whence (\ref{eq:quasi-invN}) follows. 
\end{proof}

All the remaining statements of this section are understood to hold under the same assumptions of Proposition \ref{prop:quasi-invN}. 

Zero measure sets remains of zero measure for all $t$.
\begin{lemma}\label{lemma:NullSets}
For all measurable sets $A$ such that $\r_s(A)=0$ it holds $\r_s(\Phi_t(A))=0$ for all $t\in\R$. 
\end{lemma}
\begin{proof}
Integrating \eqref{eq:yud2} we get (here we bound $p^{\sigma - 1} \leq 1$)
\be\label{eq:start-again} 
(\r_{s,N} \circ \Phi_t^N )(A)\leq (c(R,r) |t| + \r_{s,N} (A)^{\frac1p})^p\leq (c(R,r)^p |t|^p+\r_{s,N} (A)^{\frac1p}(A))2^{p-1}\,.
\ee
 Let $\delta >0$ and $\rho_{s}(A) \leq \delta$. 
Since $\r_{s,N}(A) \overset{N \to \infty}{\to} \r_{s,N}(A)$, we have  
\be 
(\r_{s,N} \circ \Phi_t^N )(A) 
\leq (c(R,r)^p |t|^p+ 2 \delta  (A)^{\frac1p}(A))2^{p-1}\, ,
\ee
for all $N$ sufficiently large (the choice of $N$ only depends on $A$).
Now letting $t_R := \frac{1}{4c(R,r)}$ we have that for all $|t| \leq t_R$: 
\be
(\r_{s,N} \circ \Phi_t^N)(A)\leq \frac12\left(2^{-p} + \delta2^{p+1}\right)\,,\quad \forall p>1\,. 
\ee
Therefore for any $\e \in (0,1/2)$ we can take $p=-\log_2\e$ and see that there is $0<\d<\e^2$ such that
\be\nonumber
\rho_{s}(A)\leq \d\quad\Rightarrow\quad(\r_{s,N} \circ \Phi_t^N )(A)\leq\e, 
\qquad |t| \leq t_R\,.
\ee
To upgrade the estimate to the limiting version 
\be\label{NullSets}
\rho_{s}(A)\leq \d\quad\Rightarrow\quad(\r_{s} \circ \Phi_t )(A)\leq\e\,, 
\qquad |t| \leq t_R\,,
\ee
we proceed as in 
\cite[Lemma 8.1]{sigma}.
This yields \eqref{NullSets} for $|t| \leq t_R$. Since $t_R$ only depends on $R$ and 
the restriction $\| u(t) \|_{H^{\beta/2}} \leq R$ is invariant under $\Phi_t$ (see \eqref{cons}) 
we can globalise to $t \in \R$ by the usual gluing procedure. 
\end{proof}

The next statement generalise the foregoing lemma to sets of positive measure. It is relevant for $0 < \mu \ll1$.

\begin{proposition}\label{lemma:radiceNLp}
Let $\sigma \in (0,1)$ be as in Proposition~\ref{prop:quasi-invN}, $\mu >0$ and $R > 0$. There exists~$C(\sigma, \mu, R, r)>0$ such that for any measurable set $A$ 
and $t \in \R$,
\be\label{eq:radiceLp}
\rho_{s}(\Phi_{t}(A))\leq 
\rho_{s}(A)^{1 - \sigma \mu^{\frac{1}{\sigma}}}   \exp \left( C(\sigma, \mu, R, r) \left( 1+  |t|^{\frac{1}{1-\sigma}} \right) \right)\,.
\ee
\end{proposition}
\begin{proof}
Due to Lemma~\ref{lemma:NullSets} we can assume $\r_{s}(A)>0$.  Consider
\be\label{eq:p}
p=\log \left( \frac{1}{2\r_{s}(A)} \right)\,,
\ee
and note that
\be\label{eq:p-consequence}
(2\r_{s}(A))^{-\frac1p} = e\,. 
\ee
Recalling 
\eqref{Def:Rho} and 
\eqref{eq:rho}
we have that $\r_{s,N}(A) \to \r_{s}(A)$ as $N\to\infty$ for all measurable $A$.
Then since $\r_{s}(A) >0$ we 
can find $\bar N = \bar N(A)$ such that $\r_{s,N}(A)\leq 2\r_{s}(A)$ for all $N>\bar N$. 

Thus, for sufficiently large $N$ (\ref{eq:quasi-invN}) reads
\bea\label{STG}
\r_{s,N}(\Phi_t^N (A))&\leq&2 \rho_{s}(A)\exp\left(p\log(1+c(R, r) |t| p^{\s-1}(2\rho_{s}(A))^{-\frac1p})\right)\nn\\
&=&2\rho_{s}(A)\exp\left(p\log(1+c(R, r) |t| p^{\s-1})\right)\nn\\
&\leq&2\rho_{s}(A)\exp\left(c(R, r) |t| p^{\s}\right)\,\nn\\
&\leq&2\rho_{s}(A)\exp\left(c(R, r) |t| \left( \ln \left( \frac{1}{2\r_{s}(A)} \right) \right)^{\s}\right)\,,
\eea
where we used \eqref{eq:p-consequence} in the second line and allowed the the constant $c(R,r)$ to increase of a fixed factor from line to line. 
For all $\mu >0$ and $\sigma \in (0,1)$ we have the Young inequality
$$
|a| |b|  \leq \sigma ( \mu |a|)^{\frac{1}{\sigma}} + (1-\sigma) \left(  \frac{|b|}{\mu} \right)^{\frac{1}{1- \sigma}}\,,
$$ 
that we will use choosing
$$
a = \left( \ln \left( \frac{1}{2\r_{s}(A)} \right) \right)^{\s} , \quad b =  c(R, r) |t|.
$$
As a consequence the \eqref{STG} gives
\begin{align}\label{STG2}
\r_{s,N}(\Phi_t^N (A)) 
& \leq 
2\rho_{s}(A)\exp\left(   \sigma \mu^{\frac{1}{\sigma}} \left( \ln \left( \frac{1}{2\r_{s}(A)}  \right) \right)   
\right) \exp \left( \frac{1-\sigma}{\mu^{\frac{1}{1-\sigma}}}  \left( c(R, r) |t| \right)^{\frac{1}{1-\sigma}} \right)
\\ \nonumber
& \qquad \qquad \leq 
\rho_{s}(A)^{1 - \sigma \mu^{\frac{1}{\sigma}}}   \exp \left( C(\sigma, \mu, R, r) \left( 1 +|t|^{\frac{1}{1-\sigma}} \right) \right)\,.
\end{align}
Then to deduce (\ref{eq:radiceLp}) one argues as in \cite[Lemma 8.1]{sigma}.
\end{proof}

Therefore we have proved that $\rho_{s}\circ\Phi_t$ is absolutely continuous w.r.t. $\rho_{s}$ with a density 
\begin{equation}\label{L1Prop}
f_{s}(t,u) \in L^1(\rho_{s}).
\end{equation}
In particular
$$ \rho_{s}\circ\Phi_t \ll \rho_{s}  \ll \g_{s}.$$

In fact $f_s$ belongs to all $L^p(\rho_{s})$ spaces. 

\begin{proposition}\label{prop_Lp}
We have 
$$
f_{s}(t,u) \in L^p(\rho_{s})
$$
for all $p \geq1$ and $t \in \R$.
\end{proposition}

\begin{proof}
Letting 
\bea
\widetilde C&:=& \exp \left( C(\sigma, \mu, R, r) \left( 1+  |t|^{\frac{1}{1-\sigma}} \right) \right) \,,\nn\\
\delta &:=& \sigma \mu^{\frac{1}{\sigma}}\label{eq:delta}\,,
\eea
we rewrite \eqref{eq:radiceLp} as
\begin{equation}\label{AlambdaBis}
\rho_{s}(\Phi_{t}(A))\leq  \widetilde C \, 
\rho_{s}(A)^{1 - \delta}   .
\end{equation}
Note that we can make $\delta$ arbitrarily small choosing $\mu$ sufficiently small. %
Let now $\l > 0$ and set 
\begin{equation}\label{Alambda}
A_{\l, t} :=\{u\,:\, f_{s}(t,u) >\l\}\,.
\end{equation}
Using \eqref{Alambda}-(\ref{L1Prop})-\eqref{AlambdaBis} we have
\begin{align}
\rho_{s} ( A_{\l,\a} ) 
&
= \frac1\l\int_{A_{\l,\a}}\l \rho_{s}  (du) 
\leq \frac1\l\int_{A_{\l,\a}}  f_{s}(t, u) \rho_{s}  (du) 
\\ \nonumber
& =\frac1\l( \rho_{s}\circ\Phi_t ) (A_{\l,\a})
\lesssim \frac{1}{\l} \widetilde C \, 
\rho_{s}(A_{\l,\a})^{1 - \delta} \,. 
\end{align}
Consequently
\be\label{eq:gammaAlambda}
\rho_{s}(A_{\l,\a})\lesssim \left(\frac{1}{\l}\right)^{1/ \d}\,.
\ee
Using \eqref{eq:gammaAlambda} we  write
$$
\| f_{s}(t,u)  \|^p_{L^p(\rho_{s})}=p\int_0^\infty \l^{p-1}\rho_{s}(A_{\l,\a})d\l
\leq p +
p\int_1^\infty \l^{p-1 - \frac{1}{\delta}}  d\l.
$$
By (\ref{eq:delta}) we see that taking $\mu >0$ sufficiently close to zero we have $p-1 - \frac{1}{\delta} < -1$ so that the integral above is finite and the statement is proved. 
\end{proof}

%%%%%%%%%%%%%%%%%%%%%%%%%%%%%%%%%%%%%%%%%%%%%%%%%%%%%%%%%%%%%%%%%%%%%%%%%%%%%%%%%%%%%%%%%%%%%%%%%%%%%%%%%%%%%%%%%%%%%%%%%%%%%%%%%%%%%%%%%%%%%%%%%%%%%%%%%%%%%%%%%%%%%%%%%%%%%%%%%%%%%%%

\section{Convergence of finite-dimensional densities for the BBM equation}\label{sec:density}

In the last section we proved that $\rho_{s}\circ \Phi_t$ is a.c. w.r.t. $\rho_{s}$ and denoted the density by $f_{s}(t,u)$. Here we complete the proof of Theorem \ref{TH:quasi}, giving an explicit expression for $f_s(t,u)$. Let
\begin{equation}\label{eq:GammaN}
\Gamma_N(u)(t):= - \frac{d}{dt}\left( \| P_N \Phi_t^Nu\|^{2r}_{H^s} + \frac12 \|P_N\Phi_t^Nu\|^2_{ H^{s+\frac\beta2}} \right)\,
\end{equation}
and set
\be\label{eq:fN}
f_{s,N}(t,u) 
:=\exp\left(\int_0^t \!\!\!\Gamma_N(u)(\tau)d\tau\right), 
\qquad \bar f_s(t,u) := f_{s,\infty}(t,u)\,. 
\ee

The following result plays a key role in the sequel.
\begin{proposition}\label{prop:convL1}
Let $\beta\in(1,2]$, $s>\frac32$ and $p \geq 1$. The sequence 
 $\{f_{s,N}(t, \cdot)\}_{N\in\N}$ converges in $L^{p}(\rho_{s})$ to $\bar{f}_s$, for all $p \geq 1$.
\end{proposition}
\begin{proof}
In the forthcoming Proposition \ref{lemma:unint} we will prove that, given $T >0$, we have 
\begin{equation}\label{UnifLp}
\sup_{N \in \mathbb{N}} 
\left\| f_{s,N}(t, \cdot) \right\|_{L^{p}(\rho_{s})}
\leq C(s,p, T), \qquad |t| \leq T.
\end{equation}
for all $p \geq 1$. This fact and the convergence in measure of $\{f_{s,N}(t, \cdot)\}_{N\in\N}$
to $\bar f_s(t, \cdot)$, defined in~\eqref{eq:fN}, guarantee that~$f_{s,N}(t, \cdot) \to \bar{f}_s(t, \cdot)$ in~$L^{p}(\rho_{s})$, fot all $p \geq 1$.
More precisely the uniform bound \eqref{UnifLp} at a fixed $p$ guarantees convergence in $L^{p'}(\rho_{s})$, fot all $p' < p$ (for the details of this classical argument, see for instance \cite[Lemma 3.7]{AIF}). 
\end{proof}

Once we have identified $\bar{f}_s$, in order to complete the proof of Theorem \ref{eq:fN}, we need to show 
that~$\bar{f}_s$ coincides with the density $f_s$ of the transport of $\rho_s$ under the flow $\Phi_t$.
This is the content of the next Proposition.

\begin{proposition}\label{Prop:density}
Let $\beta\in(1,2]$, $s>\frac32$. Then $\bar f_s(t,u)=f_s(t,u)$, $\rho_{s}$-a.s.
\end{proposition}

\begin{proof}
The assertion will follow from
\begin{equation}\label{D-T}
 \int \psi(\Phi_t u )  \rho_{s}(du) =
\int \psi(u) \bar f_s(t,u)  \rho_{s}(du)
\end{equation}
for all non negative real valued continuous $\psi$ with compact support in $H^s$.
Indeed, since $\bar f_s(t,u) \in L^1(\rho_{s})$, using the inner regularity of the measure $\rho_{s}$ and 
recalling inequality~\eqref{eq:radiceLp}
one can use \eqref{D-T} to deduce 
\begin{equation}\nonumber
 \int_{\Phi_t(E)}  \rho_{s}(du) =
\int_{E} \bar  f_s(t,u)  \rho_{s}(du),
\end{equation}
for all measurable $E$. Recalling that $f_{s}(t,u)$ is (by definition) the density of 
$\rho_{s}\circ \Phi_t$ w.r.t. $\rho_{s}$, the proof would be concluded. 

To prove \eqref{D-T}, we fix $M>0$ arbitrarily and w.l.o.g. we limit ourself to considering test functions supported within $B^s(M)$. Moreover, since the measure $\rho_s$ is supported on 
$B^{\frac\beta2}(R)$ and $\| \Phi_t u \|_{H^{\frac\beta2}} = \| u \|_{H^{\frac\beta2}}$, we can also reduce to consider test functions $\psi$ supported within $B^{\frac\beta2}(R)$. Summarising we assume
\be\label{eq:supp-psi}
\supp(\psi) \subset B^{\frac\beta2}(R)\cap B^s(M)\,.
\ee

Now, proceeding as in the first part of the proof of Proposition \ref{prop:quasi-invN}, we have for finite $N$ and for any such test function $\psi$
\begin{align}\label{D-TN}
 \int \psi(\Phi_t^N u)   \r_{s,N}(du)    =
\int \psi(u) f_{s,N}(t,u)   \r_{s,N}(du).
\end{align}

To pass to the limit as $N \to \infty$ on the l.h.s. of \eqref{D-TN} we decompose
\be\label{eq:decompose1}
\int \psi(\Phi_t^N u)   \r_{s,N}(du) = \int \psi( \Phi_t^N u)   \rho_{s}(du)+    \int \psi(\Phi_t^N u)  ( \r_{s,N} - \rho_{s}) (du)\,.
\ee
Recalling 
\eqref{Def:Rho} and 
\eqref{eq:rho}
we easily see that the second summand above goes to zero by dominated convergence.
To show that the first summand converges to 
$$
\int \psi(\Phi_t u)   \rho_{s}(du)
$$
we take $\varepsilon >0$ and we choose a compact $K = K(\varepsilon)$ of $H^s$ such that
$$
\left| \int \psi( \Phi_t^N u)   \rho_{s}(du) - \int \psi(\Phi_t u)   \rho_{s}(du) \right|
\leq \frac{\varepsilon}{2} + 
 \int_K |\psi( \Phi_t^N u)    -  \psi(\Phi_t u)|   \rho_{s}(du) ;
$$
this is possible because $\psi$ is bounded and $\rho_{s}$ is inner regular. 
To handle the second term on the r.h.s. we use the $H^{s}$ approximation of  
\cite[Proposition 2.7]{sigma}, that is for all~$\varepsilon' >0$
\begin{equation}\label{ContPreq}
\sup_{u \in K} \| \Phi_t^N u - \Phi_t u \|_{H^{s}} < \varepsilon' ,
\end{equation}
taking $N$ sufficiently large (depending on $\varepsilon'$). Since
$\psi$ is continuous, we can use \eqref{ContPreq} with $\varepsilon'$ sufficiently small in such a way that 
$$
 \int_K |\psi( \Phi_t^N u)    -  \psi(\Phi_t u)|   \rho_{s}(du)  \leq \frac{\varepsilon}{2}
$$
for all $N$ sufficiently large (depending on $\varepsilon$). This concludes the analysis of the l.h.s. of \eqref{D-TN}

To pass to the limit as $N \to \infty$ on r.h.s. of \eqref{D-TN} we decompose
\begin{align}\label{eq:decompose2}
\int \psi(u) f_{s,N}(t,u)   \r_{s,N}(du) = \int \psi(u) f_{s,N}(t,u)   \rho_{s}(du)
+    \int \psi(u) f_{s,N}(t,u)  ( \r_{s,N} - \rho_{s}) (du)\,.
\end{align}
The first addendum 
converges to 
$$
\int 
 \psi(u)  \bar{f}_{s}(t,u) \rho_s(du),
 $$
thanks to Proposition~\ref{prop:convL1}. 
To show that the second addendum converges to zero we set
$$
G_N(u):=1_{\{\|u\|_{H^\frac\beta2}\leq R\}}(u)\exp(-\|P_Nu\|^{2r}_{H^s})\,, \quad G(u):=1_{\{\|u\|_{H^\frac\beta2}\leq R\}}(u)\exp(-\|u\|^{2r}_{H^s}) \,,
$$
so that
\begin{equation}\nn
\r_{s,N}(du)=G_N(u)\gamma_s(du), \quad \r_{s}(du)=G(u)\gamma_s(du)\,.
\end{equation}
Clearly  $G_N(u)$ converges to $G(u)$ in $L^{p}(\gamma_{s})$ for every $p<\infty$.  
We rewrite the second addendum as
$$
\int 
  \psi(u) f_{s,N}(t,u)(G_N(u)-G(u)) \gamma_{s}(du)\,.
$$

We note that because of our assumption \eqref{eq:supp-psi} we have
  $$
\int |f_{s,N}(t,u)|^2 |\psi(u)|^2  \gamma_s(du) \lesssim    e^{M^{2r}}
\int |f_{s,N}(t,u)|^2  \rho_s(du) \,.
$$
Therefore it follows by the $L^2(\rho_s)$ boundedness of the sequence $\{ f_{s,N}(t,\cdot) \}_{N \in \N}$, proved in Proposition~\ref{prop:convL1} that
\begin{equation}\label{deduceThisBy}
\sup_{N \in N} \left\| \psi(u)  f_{s,N}(t,\cdot) \right\|_{L^{2}(\gamma_s)} \lesssim  e^{M^{2r}}\,.
\end{equation}
But then by H\"older's inequality
\begin{align*}
\left| \int 
  \psi(u) f_{s,N}(t,u)(G_N(u)-G(u)) \gamma_{s}(du)  \right|
& \leq  \sup_{N \in N} \left\|   \psi(u)f_{s,N}(t,u) \right\|_{L^{2}(\gamma_s)}
\|G_N(u)-G(u)\|_{L^{2}(\gamma_s)} 
\\ 
&
\lesssim e^{M^{2r}}
\|G_N(u)-G(u)\|_{L^{2}(\gamma_s)}   \overset{N \to \infty}{\to} 0 \, ,
 \end{align*}
concluding the proof. 
\end{proof}

\begin{remark}
It is wort to mention that
with some more work and keeping track on the uniformity in $t$ of the estimates in Proposition 
\ref{Prop:density}  one can show the stronger statement $$\rho_s\left(\bigcup_{|t|\leq T}\{f_s(t,u)\neq \bar f_s(t,u)\}\right)=0$$ for any $T>0$. Moreover, by similar considerations, the convergence proved in \ref{prop:convL1}
can be promoted to convergence in $C((0,T);L^p (\rho_s))$.
The same applies, after obvious modifications, to Proposition \ref{prop:convL1-NLS} and Proposition \ref{Prop:density-NLS} below.
\end{remark}

Next we state and prove the two crucial properties used in the proof of Proposition~\ref{prop:convL1}.

\begin{proposition}\label{lemma:ConvMeas}
Let $\beta \in(1,2]$, $s > 3/2$. 
Then $e^{\int_0^T\Gamma_N(u)(t)dt} \to e^{\int_0^T\Gamma(u)(t)dt}$ as $N \to \infty$ uniformly over compact subsets of $H^s$; in particular it converges in measure w.r.t.~$\rho_{s}$. 
\end{proposition}

\begin{proof}
By the continuity of the exponential function is sufficient to show 
\begin{equation}\label{MeasConv1}
\int_0^T \Gamma_N(u)(t)dt \to \int_0^T \Gamma(u)(t)dt \quad \mbox{as $N \to \infty$ uniformly on any compact subsets of $H^s$.}
\end{equation}
Since 
$$
\left|  \int_0^T \Gamma_N(t)dt - \int_0^T\Gamma(t)dt \right| \leq |T|  \sup_{t \in [0,T]} |\Gamma_N(t) -\Gamma(t)| 
$$
\eqref{MeasConv1} follows by  
\begin{equation}\label{MeasConv2}
\sup_{t \in [0,T]} |\Gamma_N(t) -\Gamma(t)|  \to 0 \quad \mbox{as $N \to \infty$ uniformly on any compact subsets of $H^s$.}
\end{equation}
%****
Since $s >3/2$ we can bound
$$
\|P_N\Phi_t^Nu\|_{W^{1,\infty}}\lesssim \| \Phi_t^Nu\|_{ H^s}\,.
$$
Let $\varepsilon >0$. 
By \eqref{H1Bis}-\eqref{H1} in Lemma \ref{lemma:boundTsu} and \cite[Propositions 2.6-2.7]{sigma} we have that for all $\varepsilon'>0$ if $u$ is in a compact set 
$K \subset B^{s}(M)$ then
\begin{equation}\label{UnifConvOnCompact}
\sup_{t \in [0,T]} |\Gamma_N(t) -\Gamma(t)| \leq C(M,T) \varepsilon'
\end{equation}
for all $N$ sufficiently large, depending on $K, \varepsilon'$. 
Thus to prove \eqref{MeasConv2} it suffices to take
$\varepsilon'  < \frac{\varepsilon}{C(M,|T|) }$.

In particular, we have convergence in measure of $e^{\int_0^T\Gamma_N(u)(t)dt}$ to $e^{\int_0^T\Gamma(u)(t)dt}$. 
Indeed, the uniform convergence of $e^{\int_0^T\Gamma_N(u)(t)dt}$ to $e^{\int_0^T\Gamma(u)(t)dt}$ on compact sets ensure that,
for all $\lambda >0$ and~$K$ compact we have  
$$
\sup_{u \in K} |e^{\int_0^T\Gamma_N(u)(t)dt} - e^{\int_0^T\Gamma(u)(t)dt}|  < \lambda, 
$$
for all sufficiently large $N$.
Thus
$$
\rho_{s} \left(  |e^{\int_0^T\Gamma_N(u)(t)dt} - e^{\int_0^T\Gamma(u)(t)dt}|  > \lambda \right) <
 \rho_{s} \left(K^C \right)\,,
$$
for all sufficiently large $N$.
Then we can choose $K$ such that $\rho_{s} \left(K^C \right)$ is
arbitrarily small by the inner regularity of $\rho_s$.
\end{proof}

We conclude proving the uniform $L^{p}(\rho_s)$ bound \eqref{UnifLp}. 
Recalling the definition \eqref{eq:fN} of $f_{s,N}$ we deduce \eqref{UnifLp} by \eqref{MaxEst1} below.

\begin{proposition}\label{lemma:unint}
Let $\beta \in(1,2]$, $s > 3/2$, $|T| >0$, $R>0$ and $p \geq 1$. We have
\begin{equation}\label{MaxEst1}
 \sup_{N \in \N} \left\| e^{\int_0^T |\Gamma_N(u)(t)|dt} \right\|_{L^{p}(\rho_{s})}
 \leq C(p,|T|, R,  s, r)\, .
\end{equation}

\end{proposition}
\begin{proof}
In order to prove \eqref{MaxEst1}, we start noting that 
the inequality \eqref{H2} and the bound \eqref{PolynomialGrowth} allow us to we estimate ($\e>0$ is arbitrarily small)
\begin{align}\nonumber
\int_0^T \left| \Gamma_N(u)(t)dt \right| & \leq |T| \sup_{t \in [0.T]} |\Gamma_N(u)(t)| 
\\
\nn
& \leq  |T| \sup_{t \in [0.T]} \| \partial_x P_N \Phi_t^N\|_{W^{1,\infty}}  \Big( \| P_N \Phi_t^N u \|^2_{H^s} +  \| P_N \Phi_t^N u\|^{2r-2}_{H^s}\|P_N  \Phi_t^N u \|^2_{H^{s-\frac\beta2}} \Big)
\\ \label{IndepOnN}
& \leq C(|T|) \|   u \|_{H^{\frac{3}{2} + \varepsilon}}  \Big( \|  u\|^2_{H^s} +  \|   u\|^{2r-2}_{H^s}\|   u\|^2_{H^{s-\frac\beta2}} \Big);
\end{align}
The constant $C(|T|)$, possibly increasing from line to line, also depends on the fixed parameters~$r,R$. Therefore
  $$
e^{\int_0^T |\Gamma_N(u)(t)|dt} \leq 
 e^ {C(|T|) \|   u \|_{H^{\frac{3}{2} + \varepsilon}}   \|  u\|^2_{H^s} } 
 e^{C(|T|) \|   u \|_{H^{\frac{3}{2} + \varepsilon}}   \|  u\|^{2r-2}_{H^s}\| u\|^2_{H^{s-\frac\beta2} } } \,
 $$
and since 
the right hand side is independent on $N$, we get by H\"older's inequality
\begin{align}
 \sup_{N \in \N} \left\| e^{\int_0^T |\Gamma_N(u)(t)|dt} \right\|_{L^{p}(\rho_{s})}
& \leq
  \left\| e^{C(|T|) \|u\|_{ H^{\frac32 + \varepsilon}}  \|u\|^2_{ H^s}}\right\|_{L^{2p}(\rho_{s})}\label{eq:norm1}\\
& \qquad \cdot  \left\| e^{C(|T|) \|u\|_{ H^{\frac32 + \varepsilon}} \|u\|^{2r-2}_{ H^s}\|u\|^2_{ H^{s-\frac\beta2}}}\right\|_{L^{2p}(\rho_{s})}
\label{eq:norm2}
\,.
\end{align}

We have ($\t:=t/C(|T|)$)
\bea
\eqref{eq:norm1}&=&\int_0^\infty dt e^{2pt} \rho_{s}\left( \|u\|_{ H^{\frac32 + \varepsilon}}  \|u\|^2_{ H^s}\geq \t\right)\label{eq:Lp1}\\
\eqref{eq:norm2}&=& \int_0^\infty dt e^{2pt} \rho_{s}\left(\|u\|_{ H^{\frac32 + \varepsilon}} \|u\|^{2r-2}_{ H^s}\|u\|^2_{ H^{s-\frac\beta2}}\geq \t \right)\label{eq:Lp2}\,.
\eea
For \eqref{eq:Lp1} we have
\bea
&&\rho_{s}\left( \|u\|_{ H^{\frac32 + \varepsilon}}  \|u\|^2_{ H^s}\geq \t\right)\nn\\
&\leq& \rho_{s}\left(\|u\|_{ H^{\frac32 + \varepsilon }}\geq \sqrt\t\right)+\rho_{s}\left(\|u\|_{ H^s}\geq \t^{\frac14}\right)\nn\\
&\leq&C\left[\exp\left(-c(R)\t^{r\frac{2s-\beta}{3 + 2 \varepsilon - \beta }}\right)+\exp\left(-\t^{\frac r2}\right)\right]\leq Ce^{-c(R)\t^{\frac r2}}\,,
\eea
where we used the forthcoming Lemma \ref{lemma:subexpXH} ($N=\infty$) along $\frac{2s-\beta}{3-\beta}>1$ for $s>\frac32$ and 
$\varepsilon >0$ sufficiently small.
Therefore (recall $r>2$)
\be\label{eq:unifint1}
\mbox{r.h.s. of } \eqref{eq:Lp1}\leq \int_0^\infty dt e^{2pt-c(R)\left(\frac{t}{T}\right)^{\frac r2}}=: C_1(p,|T|, R)<\infty\,.
\ee

We pass now to (\ref{eq:Lp2}). We split
\bea
&&\rho_{s}\left(\|u\|_{ H^{\frac32 + \varepsilon}} \|u\|^{2r-2}_{ H^s}\|u\|^2_{ H^{s-\frac\beta2}}\geq \t\right)\nn\\
&\leq& \rho_{s}\left(\|u\|_{ H^{\frac32 + \varepsilon}}\geq \t^{\kappa_1}\right)+\rho_{s}\left(\|u\|^{2r-2}_{ H^s}\geq \t^{\kappa_2}\right)\nn\\
&+&\rho_{s}\left(\|u\|^2_{ H^{s-\frac\beta2}}\geq \t^{\kappa_3}\right)\nn\\
&\leq&Ce^{-c(R) \tau^m} \, ,\nn
\eea
where we have to use Lemma \ref{lemma:subexpXH} ($N = \infty$) and
\be\label{eq:kappa}
\kappa_1+\kappa_2+\kappa_3=1\,,\quad\kappa_1>\frac{3 + 2 \varepsilon - \beta}{2r(2s-\beta)}\,,\quad
\kappa_2>1-\frac1r\,,\quad
\kappa_3>\frac1r-\frac{\beta}{r(2s-\beta)}\,
\ee
to ensure $m = m(s,r)>1$. We see by direct inspection that, taking $\varepsilon >0$ sufficiently small, the system (\ref{eq:kappa}) has always a 
solution for all $r>2$ and $\beta>1$. Thus there is $m>1$ for which
\be\label{eq:unifint2}
\mbox{r.h.s. of } \eqref{eq:Lp2}\leq \int_0^\infty dt e^{2pt-c(R)\left(\frac{t}{T}\right)^{m}}=:C_2(p,|T|, R, s, r)<\infty\,
\ee
and the proof is concluded. 
\end{proof}

The following concentration bound of the Sobolev norms for the measure $\rho_{s}$ is used in the previous proof and improves the one achieved using $\g_{s}$ (restricted to $B^{\frac\beta2}(R)$) for all $s\geq \frac\beta2(1+(r-1)^{-1})$. 

\begin{lemma}\label{lemma:subexpXH}
Let $\beta >1$, $\kappa>0$ and $\vs\leq s$. Set 
\be
a:=2r\kappa\left(\frac{2s-\beta}{2\vs-\beta}\right)\,,\qquad b:=4r\frac{s-\beta}{2\vs-\beta}\,. 
\ee
If $2\vs>\beta$ there are $C,c>0$ such that for all $N\in\N\cup\{\infty\}$ it holds
\be
\rho_{s}(\|P_Nu\|_{ H^\varsigma}\geq t^\k)\leq C\exp\left(-c\frac{t^{a}}{R^{b}}\right)\,.\label{eq:sub-expX1}
\ee
Otherwise if $2\vs\leq\beta$ the l.h.s. probability is identically zero for any $N\in\N\cup\{\infty\}$ for all $t>R^{\frac1\kappa}$.
\end{lemma}
\begin{proof}
Let $j_t$ the largest element of $\N \cup \{ \infty \}$
such that 
\begin{equation}\label{obvPreq}
2^{j (\vs-\frac\beta2)} < t^\k/R \quad \mbox{for} \quad j < j_t\,. 
\end{equation}
We split 
\be\label{eq:X++}
\|P_Nu\|^2_{ H^\vs}\leq\sum_{0\leq j < j_t} 2^{2j\vs} \|\D_jP_N u\|^2_{L^2}
 +\sum_{j \geq j_t} 2^{2j\vs} \|\D_jP_N u\|^2_{L^2} \,.
\ee
Note that if $\vs-\frac\beta2\leq 0$ and $t^\kappa>R$ then $j_t=\infty$ and the second summand above is zero. 
Thus the following bound  
$$
\sum_{0\leq j < j_t} 2^{2j\vs} \|\D_jP_N u\|^2_{L^2}  \leq \sum_{0 \leq j <  j_t} 2^{2j(\vs-\frac\beta2)} \|\D_j P_N u\|^2_{ H^{\frac\beta2}} < t^{2\k}\,, 
\quad \mbox{(use \eqref{obvPreq})}\,,
$$
holds $\rho_{s}$-a.s., therefore
\be\label{eq:X-e1}
\rho_{s}\Big(\sum_{0 \leq j < j_t} 2^{2j\vs} \|\D_jP_N u\|^2_{L^2} \geq t^{2\k}\Big)=0\,.
\ee
Thus
\be\label{eq:prob-interim}
\rho_{s}\Big(\|P_Nu\|_{ H^{\vs}}\geq t^\kappa\Big)\leq \rho_{s}\Big(\sum_{j \geq j_t} 2^{2j\vs} \|\D_jP_Nu\|^2_{L^2}\geq \frac12t^{2\kappa}\Big)\,.
\ee
Moreover we note
\be
\sum_{j \geq j_t} 2^{2j\vs} \|\D_jP_Nu\|^2_{L^2}\leq2^{-2j_t(s-\vs)}\sum_{j \geq j_t}2^{2js} \|\D_jP_Nu\|^2_{L^2}\leq 2^{-2j_t(s-\vs)}\|P_Nu\|^2_{ H^s}\,.
\ee
Therefore, assuming $\vs-\frac\beta2>0$ we have
\bea
\mbox{r.h.s. of \eqref{eq:prob-interim}}&\leq& \rho_{s}(\|P_Nu\|^2_{ H^s}\geq 2^{2j(s-\vs)-1}t^{2\kappa})\nn\\
&\leq& Ce^{-ct^{2r\kappa}2^{2rj_t(s-\vs)}}\nn\\
&\leq& C\exp\left(-ct^{2r\kappa	\left(\frac{2s-\beta}{2\vs-\beta}\right)}/R^{4r\frac{s-\beta}{2\vs-\beta}}\right)\,.
\eea
\end{proof}

%%%%%%%%%%%%%%%%%%%%%%%%%%%%%%%%%%%%%%%%%%%%%%%%%%%%%%%%%%%%%%%%%%%%%%%%%%%%%%%%%%%%%%%%%%%%%%%%%%%%%%%%%%%%%%%%%%%%%%%%%%%%%%%%%%%%%%%%%%%%%%%%%%%%%%%%%%%%%%%%%%%%%%%%%%%%%%%%%%%%%%%%%%

\section{Quasi-invariant measures for the NLS equation}\label{sect:quasiNLS}
In this section we prove the first part of Theorem \ref{TH:quasiNLS}. The proof of the formula 
\eqref{eq:DensitiesNLS} for the density is postponed to Section \ref{sec:density}.
The following result replaces \cite[Proposition 2.1]{PTV}.
\begin{lemma}\label{lemma:replace}
Let $m_0$ be given by Theorem \ref{th:PTV} and $2r>m_0+1$. Then for every $\d>0$ there is $C=C(\d,r,k)>0$ such that for all $N\in\N\cup\{\infty\}$
$$
\int_{\{\|u\|_{L^2}+\mathcal E_1(u)\leq R\}} \g_{2k}(du)e^{-R_{2k}(P_N  u)-\d \|P_N u\|^{2r}_{H^{2k-1}}} \leq C\,. 
$$
\end{lemma}
\begin{proof}
By Theorem \ref{th:PTV} $|R_{2k}(P_N  u)| \leq \| P_N  u\|^{m_0+1}_{H^{2k-1}}$. So
\bea
&&\int_{\{\|u\|_{L^2}+\mathcal E_1(u)\leq R\}}\g_{2k}(du)e^{-R_{2k}(P_N  u)-\d \|P_N  u\|^{2r}_{H^{2k-1}}}\nn\\&\leq& \int_{\{\| u\|_{L^2}+\mathcal E_1(u)\leq R\}}\g_{2k}(du)e^{\|P_N  u\|_{H^{2k-1}}^{m_0+1}-\d \|P_N  u\|^{2r}_{H^{2k-1}}}\nn
\eea
and since $2r> m_0+1$ we have the assertion. 
\end{proof}
Therefore we can define for $k\in \N$
\be\label{eq:rhoTruncatedNLS} 
\mu_{2k,N}(du):=1_{\{\|u\|_{L^2}+\mathcal E_1(u)\leq R\}}\exp(-R_{2k}(P_N  u)-\|P_Nu\|^{2r}_{H^{2k-1}}) \g_{2k}(du)\,,\quad 2r>m_0+1\,
\ee
and the sequence $\{\mu_{2k,N}\}_{N\in\N}$ has a limit~$\mu_{2k}$ (see \eqref{Def:RhoNLS}) which is the candidate for our quasi-invariant measure. Again for all $N$ the measures $\mu_{2k,N}$ and $\mu_{2k}$ depend on the parameters $R,r$, but we do not report that in the notation.

From now on we will only work with $r\gg m_0$ in the sense of Lemma \ref{lemma:replace}.

\begin{proposition}\label{prop:quasi-invN-NLS}
Let $k\in\N\setminus \{1\}$. For every measurable set $A$ and every $t\in\R$,
\be\label{eq:quasi-invN-NLS}
\mu_{2k,N}(\Phi_t^N (A))\leq \mu_{2k,N}(A)\exp\left(p\log(1+ c(R, r, k) |t| (\mu_{2k,N}(A))^{-\frac{1}{p}})\right).
\ee
\end{proposition}
%%%
\begin{proof}
Proceeding as in the proof of Proposition~\ref{prop:quasi-invN}, after a suitable change in the definition of 
the relevant objects,
 we can write for any measurable $E$  
\be\label{eq:similNLS}
\mu_{2k,N} \circ \Phi_t^N (E) = 
\int_{E} \mu_{2k,N}(du) \exp(\| P_N u \|^{2r}_{H^{2k-1}} - \|P_N \Phi_t^N u\|^{2r}_{H^{2k-1}}) 
\exp\left(\mc E_{2k}( P_N u)-\mc E_{2k}( P_N \Phi_t^N u)\right)\,;
\ee
to justify this computation we only needed that the Jacobian determinant satisfies $|\det D P_N \Phi_t^N (u)| =1$, that is a well known
consequence of the Hamiltonian structure of the quintic NLS, and the invariance
 of the constraint $1_{\{\| u \|_{L^{2}}+\mc E_1(u) \leq R\}}$ under $\Phi_t^N$, which follows by \eqref{EqNLS:ConsEnN}.

Taking $E= \Phi_{\bar t}^N A$ and using the group property of the flow
\begin{equation}\label{EqualityAtT=0-NLS}
\frac{d}{d t} \left( \mu_{2k,N} \circ \Phi_t^N (A) \right)\Big|_{t=\bar t} = 
\frac{d}{d t}\left( \mu_{2k,N} \circ \Phi_t^N (\Phi_{\bar t}^N A ) \right) \Big|_{t=0} \, .
\end{equation}
we write
\begin{align}
\frac{d}{d t} & \left( \mu_{2k,N} \circ \Phi_t^N (A) \right)\Big|_{t=\bar t}
\nn \\ 
&=-\int_{\Phi^N_{\bar t}(A)} \mu_{2k,N}(du) \left(r \|P_Nu\|^{2r-2}_{H^{2k-1}}\frac{d}{dt}\|P_N \Phi_t^Nu\|^{2}_{H^s}\Big|_{t=0}+\frac{d}{dt}
\mc E_{2k}( P_N \Phi_t^N u)\Big|_{t=0}\right). \label{eq:continuaNLS}
\end{align}

The derivative can be therefore estimated using Proposition \ref{prop:NLS-dtH} for the first summand and Theorem \ref{th:PTV} for the second one. We have
\be
|\eqref{eq:continuaNLS}| \lesssim \int_{\Phi^N_{\bar t}(A)} \mu_{2k,N}(du) (1+ \|P_Nu\|_{H^{2k-1}})^{2r}\,,
\ee
where we used $r \gg m_0$. By H\"older's inequality
\begin{align}
& \int_{\Phi^N_{\bar t}(A)} \mu_{2k,N}(du) (1+ \|P_Nu\|_{H^{2k-1}})^{2r} 
\leq \left(\int \mu_{2k,N}(du) (1+\|P_Nu\|_{H^{2k-1}})^{2rp}\right)^{\frac1p} (\mu_{2k,N}(\Phi^N_{\bar t}(A)))^{1-\frac1p}\nn\\
&\leq  \!\!\left(\sup_{x\geq1} 2 x^{2r}e^{-\frac{x^{2r}}{2p}}\right)\left(\int_{\{\|u\|_{L^2}+\mathcal E_1(u)\leq R\}}\g_{k}(du)e^{-R_{2k}(P_N u)-\frac12\| P_N u\|^{2r}_{H^{2k-1}}} \right)^{\frac1p} (\mu_{2k,N}(\Phi^N_{\bar t}(A)))^{1-\frac1p}\nn\\
&\lesssim p(\mu_{2k,N}(\Phi^N_{\bar t}(A)))^{1-\frac1p}\,,
\end{align}
where the last inequality is due to Lemma \ref{lemma:replace}.
Therefore we conclude that
\be\label{eq:yud1NLS}
\frac{d}{d t} \left( \mu_{2k,N} \circ \Phi_t^N (A) \right)\leq c(R, r,k) p(\mu_{2k,N}(\Phi^N_{t}(A)))^{1-\frac1p}\,,
\ee
whence
\be\label{eq:yud2NLS}
\frac{d}{d t} \left( \mu_{2k,N} \circ \Phi_t^N (A) \right)^{\frac1p}\leq c(R, r,k) \,,
\ee
which gives (\ref{eq:quasi-invN-NLS}).

\end{proof}

Once we achieved the estimate (\ref{eq:quasi-invN-NLS}), we can take a well paved route to prove the quasi-invariance of $\mu_{2k}$ under $\Phi_t$. We shall only state the main steps of the proof without proofs, which can be directly adapted for instance from Section~\ref{sect:quasi} or our previous paper \cite{gauge}. 

As we noted in the proof of Lemma~\ref{lemma:NullSets} (which works also for $\sigma =1$), 
the first outcome of a bound like \eqref{eq:quasi-invN-NLS} 
is that zero measure sets remains of zero measure for all $t \in \R$, namely
 
\begin{lemma}\label{lemma:NullSetsNLS}
For all measurable sets $A$ such that $\mu_{2k}(A)=0$ it holds $\mu_{2k}(\Phi_t(A))=0$ for all $t\in\R$. 
\end{lemma}

The next statement is proved as \cite[Lemma 3.3]{gauge}. Note that this is weaker than Proposition~\ref{lemma:radiceNLp}, since 
we use the estimate \eqref{eq:quasi-invN-NLS} in place of the stronger estimate \eqref{eq:quasi-invN}.    

\begin{proposition}\label{lemma:radiceNLp-NLS}
There exists~$C(R, r, k)>0$ such that for any measurable set $A$ one has 
\be\label{eq:radiceLp-NLS}
\mu_{2k}(\Phi_{t}(A))\leq \mu_{2k}(A)^{e^{-|t|C(R, r, k) }}\,.
\ee
\end{proposition}

Therefore we have proved that $\mu_{2k}\circ\Phi_t$ is absolutely continuous w.r.t. $\mu_{2k}$ and
 thereby w.r.t.~$\g_{2k}$. Let us denote the density of the transported measure $f_{2k}(t,u) \in L^1(\mu_{2k})$.
An important remark is that $f_{2k}$ has slightly more integrability. 

\begin{proposition}\label{prop_Lp-NLS}
For any $t\in\R$ set $p=p(t)=(1-e^{-|t| C(R,r,k)})^{-1}>1$. Then $f_{2k}(t,u) \in L^p(\mu_{2k})$. 
\end{proposition}
The proof is done as in \cite[Proposition 3.4]{gauge}

%%%%%%%%%%%%%%%%%%%%%%%%%%%%%%%%%%%%%%%%%%%%%%%%%%%%%%%%%%%%%%%%%%%%%%%%%%%%%%%%%%%%%%%%%%%%%%%%%%%%%%%%%%%%%%%%%%%%%%%%%%

\section{Convergence of finite-dimensional densities for the NLS equation}\label{DensityNLS}

In Section \ref{sect:quasiNLS} we proved that $\mu_{2k}\circ \Phi_t$ is a.c. w.r.t. $\mu_{2k}$ and denoted the density by $f_{2k}(t,u)$. Here we give an explicit expression for $f_{2k}(t,u)$. Let
\begin{equation}\label{eq:GammaN-NLS}
\Gamma_N(u)(t):=  - \frac{d}{dt}\left( \| P_N \Phi_t^Nu\|^{2r}_{H^{2k-1}} + \mc E_{2k}(P_N\Phi_t^Nu) \right)\,
\end{equation}
and set
\be\label{eq:fN-NLS}
f_{2k,N}(t,u):=\exp\left(\int_0^t \!\!\!\Gamma_N(u)(\tau)d\tau\right), 
\qquad \bar f_{2k}(t,u) := f_{2k,\infty}(t,u)\,. 
\ee

The key fact on the sequence $\{f_{2k,N}(t, \cdot)\}_{N\in\N}$ is that for any $t\in[0,1)$ it 
converges in $L^{p}(\mu_{2k})$ for some $p>1$. 

\begin{proposition}\label{prop:convL1-NLS}
Let $k\geq2$ be an integer, $0<\overline {T}\ll1$ (as in Lemma \ref{lemma:unint-NLS}), $T\in(0,\overline {T})$. There is $p(\overline{T})=p>1$ such that for all $|t|<T$ the sequence 
 $\{f_{2k,N}(t, \cdot)\}_{N\in\N}$ converges in $L^{p}(\mu_{2k})$ to $\bar{f}_{2k}$.
\end{proposition}
Proposition \ref{prop:convL1-NLS} is a consequence of the uniform $L^p(\mu_{2k})$ bound proved in Lemma 
\ref{lemma:unint-NLS}. We noted this already in the case of the BBM equation (see Proposition \ref{prop:convL1}). 
However the uniform $L^p(\mu_{2k})$ estimates of  Lemma
\ref{lemma:unint-NLS} holds for small times, hence the same holds for Proposition \ref{prop:convL1-NLS}.  Then we obtain   the following statement.
\begin{proposition}\label{Prop:density-NLS}
Let $k\in\N\setminus\{1\}$.
There exists $0<\overline {T}\ll1$ such that $\bar f_{2k}(t,u)=f_{2k}(t,u)$ $\mu_{2k}$-a.s. for all $|t|<\overline {T}$. 
\end{proposition}
To show the foregoing statement we proceed as in the proof of Proposition \ref{Prop:density}. The only difference is that since 
Proposition \ref{prop:convL1-NLS} holds for small times, also
Proposition \ref{Prop:density-NLS} is local in time.

\begin{proof}
It suffices to prove that, given $\overline T$ as in Lemma \ref{lemma:unint-NLS}, for all $|t|<\overline T$ and for all non negative real valued continuous $\psi$ with compact support in $H^{2k-\frac12-}$
\begin{equation}\label{D-TBis}
 \int \psi(\Phi_t u )  \mu_{2k}(du) =
\int \psi(u) \bar f_{2k}(t,u)  \mu_{2k}(du).
\end{equation}
Indeed, since for such $t$, $\bar f_{2k}(t,u) \in L^1(\mu_{2k})$, by the inner regularity of $\mu_{2k}$, \eqref{eq:radiceLp-NLS}
and \eqref{D-TBis} we obtain
\begin{equation}\nonumber
 \int_{\Phi_t(E)}  \mu_{2k}(du) =
\int_{E} \bar  f_{2k}(t,u)  \mu_{2k}(du),
\end{equation}
for all measurable $E$. Recalling that $f_{2k}(t,u)$ is (by definition) the density of 
$\mu_{2k}\circ \Phi_t$ w.r.t. $\mu_{2k}$, the proof would be concluded. 

We can consider only test functions $\psi$ supported within $B^{2k-\frac12-}(M)$ for a given $M>0$:
\be\label{eq:supp-psi_NLS}
\supp(\psi) \subset \{u\,:\,\|u\|_{L^2}+\mathcal E_1(u)\leq R\}\cap B^{2k-\frac12-}(M)\,.
\ee
Then
\begin{align}\label{D-TNBis}
 \int \psi(\Phi_t^N u)   \mu_{2k,N}(du)    =
\int \psi(u) f_{2k,N}(t,u)   \mu_{2k,N}(du)\,.
\end{align}
Write the l.h.s. of \eqref{D-TNBis} as
\bea
\int \psi(\Phi_t^N u)   \mu_{2k,N}(du) &=& \int \psi( \Phi_t u)   \mu_{2k}(du)\nn\\
&&+
\int \psi( \Phi_t^N u)   \mu_{2k}(du)
-\int \psi(\Phi_t u)   \mu_{2k}(du)
\label{eq:decompose0}\\
&&+ \int \psi(\Phi_t^N u)  ( \mu_{2k,N} - \mu_{2k}) (du).
\label{eq:decompose1Bis}
\eea
For the term (\ref{eq:decompose0}) we take any $\varepsilon >0$ and $K = K(\varepsilon)$ such that
$$
\left| \int \psi( \Phi_t^N u)   \mu_{2k}(du) - \int \psi(\Phi_t u)   \mu_{2k}(du) \right|
\leq \frac{\varepsilon}{2} + 
 \int_K |\psi( \Phi_t^N u)    -  \psi(\Phi_t u)|   \mu_{2k}(du) ;
$$
By Lemma \ref{lemma:NLS-diff-flows} and the continuity of $\psi$ to get for all $\e>0$
$$
 \int_K |\psi( \Phi_t^N u)    -  \psi(\Phi_t u)|   \mu_{2k}(du)  \leq \frac{\varepsilon}{2}
$$
for all $N$ sufficiently large (depending on $\varepsilon$). The term (\ref{eq:decompose1Bis}) goes to zero by dominated convergence. This concludes the analysis of the l.h.s. of \eqref{D-TNBis}

On r.h.s. of \eqref{D-TNBis} we decompose
\begin{align}\label{eq:decompose2Bis}
\int \psi(u) f_{2k,N}(t,u)   \mu_{2k,N}(du) = \int \psi(u) f_{2k,N}(t,u)   \mu_{2k}(du)
+    \int \psi(u) f_{2k,N}(t,u)  ( \mu_{2k,N} - \mu_{2k}) (du)\,.
\end{align}
By Proposition~\ref{prop:convL1-NLS} the first addendum 
converges to 
$$
\int 
 \psi(u)  \bar{f}_{2k}(t,u) \mu_{2k}(du)\,.
 $$
We set
\bea
G_N(u)&:=&1_{\{ \|u\|_{L^2}+\mathcal E_1(u) \leq R\}}(u)\exp(-R_{2k}(P_Nu)-\|P_Nu\|^{2r}_{H^{2k-1}})\,,\\
G(u)&:=&1_{\{ \|u\|_{L^2}+\mathcal E_1(u) \leq R\}}(u)\exp(-R_{2k}(u)-\|u\|^{2r}_{H^{2k-1}}) \,,
\eea
so that
\begin{equation}\nn 
\mu_{2k,N}(du)=G_N(u)\gamma_{2k}(du), \quad \mu_{2k}(du)=G(u)\gamma_{2k}(du)\,
\end{equation}
and we can rewrite the second summand in (\ref{eq:decompose2Bis}) as
$$
\int 
  \psi(u) f_{2k,N}(t,u)(G_N(u)-G(u)) \gamma_{2k}(du)\,.
$$
The crucial point here is that the bound (\ref{eq:LipRk}) (set $v=0$) for $R_{2k}(P_Nu)$ and $m_0\ll r$ entail that there is a $C>0$ (this constant will vary during the proof) such that
\be\label{eq:boundG-NLS}
\exp(-R_{2k}(P_Nu)-\|P_Nu\|^{2r}_{H^{2k-1}})\leq C
\ee
for all $N\in\N$ uniformly for $u\in \supp(\psi)$ (recall \eqref{eq:supp-psi_NLS}). This yields immediately
$G_N(u) \to G(u)$ in $L^{p}(\gamma_{2k})$ for all $p\geq1$ (see Lemma \ref{lemma:replace}). 
Moreover, bearing in mind \eqref{eq:supp-psi_NLS} we have that for a suitable $p>1$ given by subsequent Lemma \ref{lemma:unint-NLS}
$$
\sup_{N\in\N}\int |f_{2k,N}(t,u)|^p |\psi(u)|^p  \gamma_{2k}(du) \leq e^{CM^{2r}}\sup_{N\in\N}
\int |f_{s,N}(t,u)|^p  \mu_{2k}(du)
\lesssim e^{M^{2r}}\,.
$$
By H\"older's inequality ($q$ being the H\"older conjugate of $p$)
\begin{align*}
\left| \int 
  \psi(u) f_{2k,N}(t,u)(G_N(u)-G(u)) \gamma_{2k}(du)  \right|
& \lesssim  \sup_{N \in N} \left\|  \psi(u) f_{2k,N}(t,\cdot) \right\|_{L^{p}(\gamma_{2k})}
\|G_N(u)-G(u)\|_{L^{q}(\gamma_{2k})} 
\\ 
&
\lesssim  e^{M^{2r}}
\|G_N(u)-G(u)\|_{L^{q}(\gamma_{2k})}   \overset{N \to \infty}{\to} 0 \, .
 \end{align*}
So the second summand in (\ref{eq:decompose2Bis}) vanishes in the limit $N\rightarrow\infty$ and the proof is concluded. 
\end{proof}
Next we globalise this statement as follows. Let $|T|<\overline T$. We have to prove that for any measurable set $A$ and all $m\in\N$
\be\label{eq:global-stat}
\mu_{2k}(\Phi_{mT}(A))=\int_A \mu_{2k}(du)\exp\left(- \mc E_{2k}(\Phi_{mT} u) 
+ \mc E_{2k}( u) - \|\Phi_{mT}u\|^{2r}_{H^{2k-1}} + \|u\|^{2r}_{H^{2k-1}}\right)\,.
\ee
Since $T$ is arbitrarily chosen, we will have that for any measurable set $A$ and any $t\in\R$
\be\label{eq:global-stat-t}
\mu_{2k}(\Phi_{t}(A))=\int_A \mu_{2k}(du)\exp\left(- \mc E_{2k}(\Phi_{mT} u) 
+ \mc E_{2k}( u) -\|\Phi_{t}u\|^{2r}_{H^{2k-1}}+\|u\|^{2r}_{H^{2k-1}}\right)\,.
\ee
We prove \eqref{eq:global-stat} by induction over $m$. The case $m=1$ follows from Proposition \ref{Prop:density-NLS}. Assume now \eqref{eq:global-stat} holds for $m-1$. We have
\bea
\mu_{2k}(\Phi_{mT}(A))&=&\mu_{2k}(\Phi_{(m-1)T}(\Phi_T(A)))\nn\\
&=&\int_{\Phi_T(A)} \mu_{2k}(du)\exp\left(-\mc E_{2k}(\Phi_{(m-1)T} u) +\mc E_{2k}(u) -\|\Phi_{(m-1)T}u\|^{2r}_{H^{2k-1}}+\|u\|^{2r}_{H^{2k-1}}\right)\nn\\
&=&\int_{A} (\mu_{2k}\circ\Phi_T) (du)\exp\left(-\mc E_{2k}(\Phi_{mt} u) + \mc E_{2k}(\Phi_T u) -\|\Phi_{mT}u\|^{2r}_{H^{2k-1}}+\|\Phi_Tu\|^{2r}_{H^{2k-1}}\right)\nn\\
&=&\int_A \mu_{2k}(du)\exp\left(- \mc E_{2k}(\Phi_{mt} u) + \mc E_{2k}( u) -\|\Phi_{mT}u\|^{2r}_{H^{2k-1}}+\|u\|^{2r}_{H^{2k-1}}\right)\,.\nn
\eea
We used the induction assumption in the second identity and in the last identity the cancellation at the exponent is again due to Proposition \ref{Prop:density-NLS}. This concludes the proof of Theorem \ref{TH:quasiNLS}.

It remains to prove Proposition \ref{prop:convL1-NLS}. To this end we need two accessory results. 

\begin{lemma}\label{lemma:unint-NLS}
Let $k\geq2$ an integer and $p \geq 1$. 
There exists $\overline {T} = \overline{T}(p,k,r,R)$  
such that for all~$|T|<\overline {T}$
\begin{equation}\label{MaxEst1-NLS}
 \sup_{N \in \N} \left\| e^{\int_0^T |\Gamma_N(u)(t)|dt} \right\|_{L^{p}(\mu_{2k})}
 \leq C(p,k,r, R, \overline{T})\,.
\end{equation}
\end{lemma}
\begin{proof}
We will need the standard estimate 
\be\label{ExpGrowth}
\|P_N \Phi_T^N u \|_{H^{2k-1}}\leq e^{C_R |T| }  \|P_Nu\|_{H^{2k-1}}, 
\qquad \|u \|_{L^2} + \mc E_{1}(u)\leq R
\ee
that can be deduced, for instance, from Proposition \ref{prop:NLS-dtH}.
Then, recalling \eqref{eq:GammaN-NLS}
we can bound 
\begin{equation}\nn
\left| \Gamma_N(u)(t)dt \right| 
 \leq r \| P_N \Phi_t^Nu\|_{H^{2k-1}}^{2r - 2} \left|
\frac{d}{dt}\| P_N \Phi_t^Nu\|_{H^{2k-1}}^{2} \right| + \left| \frac{d}{dt} \mc E_{2k}(P_N\Phi_t^Nu) \right|
 \leq C_{r,R,k}   \| P_N \Phi_t^Nu\|_{H^{2k-1}}^{2r}  
\end{equation}
as long as $\|u \|_{L^2} + \mc E_{1}(u)\leq R$;
we used \eqref{DtH2Growth} and \eqref{1Smoothing} together with $r \gg m_0$ 
in the second inequality. Thus, using \eqref{ExpGrowth} we arrive to
\begin{equation}\nn
\left| \Gamma_N(u)(t)dt \right|  \leq C_{r,R,k} e^{C_{r,R} |T|}   \| P_N u\|_{H^{2k-1}}^{2r}
\end{equation}
and so
\be
\int_0^T \left| \Gamma_N(u)(t)dt \right| \leq |T|  \sup_{t \in [0,T]} |\Gamma_N(u)(t)| \leq
|T| C_{r,R,k} e^{C_{r,R} |T|} \| P_N u\|_{H^{2k-1}}^{2r}
\ee
and   
$$
e^{p \int_0^T |\Gamma_N(u)(t)|dt} \leq e^{p |T| C_{r,R,k} e^{C_{r,R} |T|}  \|P_Nu\|_{H^{2k-1}}^{2r}} 
\leq e^{p \overline{T} C_{r,R,k} e^{C_{r,R} \overline{T}}  \|u\|_{H^{2k-1}}^{2r}}  \,,
 $$
 as long as $|T|\leq \overline{T}$ and $\|u \|_{L^2} + \mc E_{1}(u)\leq R$.
Since the last condition is satisfied in the 
support of~$\mu_{2k}$ and since the  right hand side is independent on $N$, 
we can deduce  that
\begin{align}
 \sup_{N \in \N} &  \left\| e^{\int_0^T |\Gamma_N(u)(t)|dt} \right\|_{L^{p}(\mu_{2k})}^p
\\ \nn
& 
\leq \int_{\{  \| u \|_{L^{2}}+\mc E_1(u) \leq R \}}  \g_{2k}(du) e^{p \overline{T} C_{r,R,k} e^{C_{r,R} \overline{T}}  \|u\|_{H^{2k-1}}^{2r}} 
e^{-R_{2k}( u) - \| u\|^{2r}_{H^{2k-1}}}  
\leq C(p,k,r,R, \overline{T})
\end{align}
as long as 
$$
p \overline{T} C_{r,R,k} e^{C_{r,R} \overline{T}}  <1.
$$
This completes the proof of Lemma~\ref{lemma:unint-NLS}.
\end{proof}
\begin{lemma}\label{lemma:ConvMeasNLS}
Let $k\geq2$ an integer and $T >0$. 
Then $e^{\int_0^T\Gamma_N(u)(t)dt} \to e^{\int_0^T\Gamma(u)(t)dt}$ as $N \to \infty$ uniformly over compact subsets of $H^s$; in particular it converges in measure w.r.t.~$\mu_{2k}$. 
\end{lemma}
\begin{proof}
By the continuity of the exponential function is sufficient to show 
\begin{equation}\label{MeasConv1NLS}
\int_0^T \Gamma_N(u)(t)dt \to \int_0^T \Gamma(u)(t)dt \quad \mbox{as $N \to \infty$ uniformly on  compact subsets of $H^{2k-1}$.}
\end{equation}

Combining (\ref{H1NLS}) in Lemma \ref{lemma:boundTsuNLS} and Lemma \ref{lemma:tentative} we have
for some $0 \mu_1 \leq \mu_2 <2$ and for arbitrary $t$

\bea   
|\Gamma_N(t) -\Gamma(t)|&\leq& 
C(r,R,k)(\|\Phi_tu\|^{r-1}_{H^{2k-1}}-\| P_N \Phi^N_tu\|^{r-1}_{H^{2k-1}})(1+\|\Phi_tu\|^{\ell}_{H^{2k-1}}+\| P_N \Phi^N_tu\|^{\ell}_{H^{2k-1}})\nn\\
&+& C(r,R,k) \sum_{j = 1,2} \|\Phi_tu -P_N \Phi^N_tu \|^{\mu_j}_{H^{2k-1}}(1+\|\Phi_tu\|^{\ell}_{H^{2k-1}}+\| P_N \Phi^N_tu\|^{\ell}_{H^{2k-1}})\nn\\
&+&C(m_0,R,k)(\|\Phi_tu -  P_N \Phi^N_tu\|_{H^{2k-1}})(1+\|\Phi_tu\|^{m_0}_{H^{2k-1}}+\| P_N \Phi^N_tu\|^{m_0}_{H^{2k-1}})\nn\,. %\label{MeasConv2NLS}\,.
\eea
Fix now $\e>0$ arbitrarily small. The above formula together with Lemma \ref{lemma:NLS-diff-flows} yields 
\be
\sup_{t\in[0,T]}|\Gamma_N(t) -\Gamma(t)|<\frac{\e}{T}
\ee
for $N$ large enough, uniformly on compact subsets of $H^{2k-1}$. Since 
$$
\left|  \int_0^T \Gamma_N(t)dt - \int_0^T\Gamma(t)dt \right| \leq T \sup_{t \in [0,T]} |\Gamma_N(t) -\Gamma(t)|\,,
$$
\eqref{MeasConv1NLS} follows. Convergence in measure can be deduced as in Lemma \ref{lemma:ConvMeas}. 
\end{proof}

%%%%%%%%%%%%%%%%%%%%%%%%%%%%%%%%%%%%%%%%%%%%%%%%%%%%%%%%%%%%%%%%%%%%%%%%%%%%%%%%%%%%%%%%%%%%%%%%%%%%%%%%%%%%%%%%%%%%%%%%%%%%%%%%%%%%%%%%%%%%%%%%%%%%%%%%%%%%%%%%%%%%%%%%%%%%%%%%%%%%%%%


\begin{thebibliography}{10}
%
\bibitem{BCD} H.~Bahouri, J.-Y.~Chemin,  R.~Danchin,  {\em Fourier analysis and nonlinear partial differential equations,}
Grundlehren der Mathematischen Wissenschaften [Fundamental Principles of Mathematical Sciences],  343. Springer, Heidelberg, 2011. xvi+523 pp. 

\bibitem{B94} J. Bourgain,  {\em Periodic nonlinear Schr\"odinger equation and invariant measures}, Commun. Math. Phys. 166 (1994), 1-26.
%
\bibitem{BourgainBook} J. Bourgain, {\em Global solutions of nonlinear Schrödinger equations}. 
American Mathematical Society Colloquium Publications, 46.
%
\bibitem{BT} N. Burq, L. Thomann, {\em Almost sure scattering for the one dimensional nonlinear Schr\"odinger equation},
arXiv:2012.13571.
%
\bibitem{ABC1} A. B. Cruzeiro, {\em Equations diff\'erentielles ordinaires: non-explosion et measures quasi invariantes}, J. Funct. Anal. 54 (1983),  193-205.
%
\bibitem{ABC2} A. B. Cruzeiro, {\em Equations diff\'erentielles sur l'espace de Wiener et formules de Cameron-Martin non lin\'eaires}, J. Funct. Anal. 54 (1983),  206-227.
%
\bibitem{CM}  R.~Cameron,  W.~ Martin,  {\em Transformations of Wiener Integrals under Translations}. Annals of Mathematics. 45 (1944),  386--396.
%
\bibitem{deb} A. Debussche, Y. Tsustumi, {\em Quasi-Invariance of Gaussian Measures Transported by the Cubic NLS with Third-Order Dispersion on $\T$}, arXiv:2002.04899 (2020). 
%
\bibitem{DTV}  Y.~Deng,  N.~Tzvetkov, N.~Visciglia, {\em Invariant measures and long time behaviour for the Benjamin-Ono equation III},  Comm. Math. Phys. 339 (2015), 815--857. 
%
\bibitem{FS} J.~Forlano, K.~Seong, {\em  Transport of Gaussian measures under the flow of one-dimensional fractional nonlinear Schr\"odinger equations}, arXiv:2102.13398 [math.AP] (2021).
%%
\bibitem{forl}  J. Forlano, W. Trenberth, {\em On the transport of Gaussian measures under the one-dimensional fractional nonlinear Schr\"odinger equation}, Ann. Inst. Henri Poincare, Anal. Non Lineaire, 36 (2019), 1987-2025.
%
\bibitem{gauge} G. Genovese, R. Luc\`a, N. Tzvetkov, {\em Quasi-invariance of low regularity Gaussian measures under the gauge map of the periodic derivative NLS}, Journal of Functional Analysis 282 (2022) 109263. 
%
\bibitem{BBM2} G. Genovese, R. Luc\`a, N. Tzvetkov, {\em Quasi-invariance of Gaussian measures for the periodic Benjamin-Ono-BBM equation}, Stoch PDE: Anal Comp (2022). https://doi.org/10.1007/s40072-022-00240-2 282 (2022). 
%  
\bibitem{GLV0}
G. Genovese, R. Luc\`a, D, Valeri,
{\em  Gibbs Measures Associated to the Integrals of Motion of the Periodic DNLS}, Sel. Math. New Ser., 22(3), 1663-1702, (2016).
%
\bibitem{GLV}
G. Genovese, R. Luc\`a, D, Valeri,
{\em Invariant measures for the periodic derivative nonlinear Schr\"odinger equation},
Mathematische Annalen 374 (3-4), 1075-1138, 2019. 
%
\bibitem{girs} I.~Girsanov,  {\em On transforming a certain class of stochastic processes by absolutely continuous substitution of measures}. 
Theory of Probability and its Applications. 5  (1960)  285--301.
%
\bibitem{GOTW} T. S. Gunaratnam, T. Oh, N. Tzvetkov, H. Weber, {\em Quasi-invariant Gaussian measures for the nonlinear wave equation in three dimensions}, arXiv:1808.03158 (2018).
%
\bibitem{KenigPilod} C. E. Kenig and D. Pilod, {\em Local well-posedness for the KdV hierarchy at high regularity}, Adv. Differential Equations, 21, (2016), 801-836.
%
\bibitem{Mammeri} Y. Mammeri, {\em Long time bounds for the periodic Benjamin--Ono-BBM equation}, Nonlinear Anal. 71 (2009), 5010--5021.
%
\bibitem{OS} T. Oh, K. Seong,  {\em Quasi-invariant Gaussian measures for the cubic fourth order nonlinear Schr\"odinger equation in negative Sobolev spaces}, 	arXiv:2012.06732 [math.AP] (2020).
%
\bibitem{OT1} T. Oh, N. Tzvetkov, {\em Quasi-invariant Gaussian measures for the cubic fourth order nonlinear Schr\"odinger equation}, Probab. Theory Relat. Fields 169 (2017), 1121-1168.
%
\bibitem{OT2} T. Oh, N. Tzvetkov, {\em Quasi-invariant Gaussian measures for the two-dimensional defocusing cubic nonlinear wave equation}, JEMS 22 (2020) 1785--1826.
%
\bibitem{OT3} T. Oh, P. Sosoe, N. Tzvetkov, {\em An optimal regularity result on the quasi-invariant Gaussian measures for the cubic fourth order nonlinear Schr\"odinger equation}, J. Ec. Polytechnique, Math., 5 (2018), 793-841.
%
\bibitem{OTT} T. Oh, Y. Tsutsumi and N. Tzvetkov, {\em Quasi-invariant Gaussian measures for the cubic nonlinear Schr\"odinger equation with third-order dispersion}, C. R. Acad. Sci. Paris, Ser. I, 357 (2019), 366-381.
%
\bibitem{PTV2} F. Planchon, N. Tzvetkov, N. Visciglia, On the growth of Sobolev norms for NLS on 2-and
3-dimensional manifolds, Anal. PDE 10 (2017), 1123-1147.
%
\bibitem{PTV} F. Planchon, N. Tzvetkov and N. Visciglia, {\em Transport of Gaussian measures by the flow of the nonlinear Schr\"odinger equation}, Math. Ann. 378 (2020) 389--423.
%
\bibitem{Ramer} R.~Ramer, {\em On nonlinear transformations of Gaussian measures}, Journal of Functional Analysis 15 (1974), 166-187.

\bibitem{phil} P. Sosoe, W. J. Trenberth and T. Xiao, {\em Quasi-invariance of fractional Gaussian fields by nonlinear wave equation with polynomial nonlinearity}, Differential Integral Equations 33 (2020) 393-430.
%
\bibitem{AIF} N. Tzvetkov, Invariant measures for the defocusing Nonlinear Schr\"odinger equation, {\em Ann. Inst. Fourier}, 58 (2008) no. 7, pp. 2543-2604.
%
\bibitem{sigma} N. Tzvetkov, {\em Quasi-invariant Gaussian measures for one dimensional Hamiltonian PDEs}, Forum Math. Sigma 3 (2015), e28, 35 pp.
%
\bibitem{TV13a} N. Tzvetkov, N. Visciglia,  {\em Gaussian measures associated to the higher order conservation laws of the Benjamin-Ono equation},  Ann. Sci. ENS 46 (2013) 249--299.

\bibitem{TV13b}  N.~Tzvetkov,  N.~Visciglia, {\em Invariant measures and long time behaviour for the Benjamin-Ono equation}, Int. Math. Res. Not. 17 (2014) 4679--4614.
%
\bibitem{TV14} N.~Tzvetkov,  N.~Visciglia, {\em  Invariant measures and long time behaviour for the Benjamin-Ono equation II},  J. Math. Pures Appl. 103 (2015), 102--141.
%
\bibitem{ust}  A.S.~ \"Ust\"unel,  M.~Zakai,  {\em Transformation of measure on Wiener space}, Springer Science \& Business Media, (2013).

\end{thebibliography}
\end{document}